\pdfoutput=1 		
\documentclass[leqno]{article}
\usepackage[normal]{phaine_style}
\usepackage[margin=1.5in]{geometry}
\usepackage{xparse}
\include{standard_macros}

\renewcommand{\orientedtimes}{\mathbin{\overleftarrow{\times}}}

\newcommand{\ab}{\mathrm{ab}}

\renewcommand{\lex}{\mathrm{lex}}
\newcommand{\acc}{\mathrm{acc}}
\newcommand{\Funlexacc}{\Fun^{\lex,\acc}}

\renewcommand{\Fun}{\mathrm{Fun}}
\renewcommand{\cts}{\mathrm{cts}}
\renewcommand{\Gal}{\mathrm{Gal}}
\renewcommand{\Functs}{\Fun^{\cts}}

\newcommand{\Ani}{\categ{Ani}}
\renewcommand{\Spc}{\Ani}
\DeclareMathOperator{\Cond}{Cond}
\newcommand{\ICond}{\categ{Cond}}

\newcommand{\CondAni}{\Cond(\Ani)}
\newcommand{\CondSet}{\Cond(\Set)}

\newcommand{\CondCat}{\Cond(\Catinfty)}

\newcommand{\CondGrp}{\Cond(\Grp)}

\newcommand{\Comp}{\categ{Comp}}

\DeclareMathOperator{\FinSet}{\Set_{\fin}}

\DeclareMathOperator{\ProFin}{\Pro(\FinSet)}

\newcommand{\Extr}{\categ{Extr}}
\newcommand{\Extrop}{\Extr^{\op}}

\DeclareMathOperator{\wc}{wc}
\newcommand{\Affwc}{\Aff{}^{\kern0.15em\wc}}
\newcommand{\Affwcop}{\Aff{}^{\kern0.15em\wc,\op}}

\newcommand{\aff}{\ensuremath{\textup{aff}}}

\newcommand{\RTop}{\categ{RTop}}

\newcommand{\cond}{\mathrm{cond}}

\newcommand{\Bcond}{\Bup^{\cond}}
\newcommand{\BcondGal}{\Bup^{\cond}\mathrm{Gal}}

\newcommand{\prodisccompl}{^{\wedge}_{\disc}}

\DeclareMathOperator{\BGal}{BGal}

\newcommand{\Piet}{\Pi_{\infty}^{\et}}
\newcommand{\Pietprotrun}{\Pi_{<\infty}^{\et}}
\newcommand{\Pietprofin}{\widehat{\Pi}_{\infty}^{\et}}

\newcommand{\Picond}{\Pi_{\infty}^{\cond}}
\newcommand{\CondShape}[1]{\Picond(#1)}
\newcommand{\picond}{\uppi^{\cond}}

\newcommand{\pizerocond}{\picond_0}

\newcommand{\pionecond}{\picond_1}

\newcommand{\pioneet}{\uppi_1^{\et}}

\newcommand{\Zpos}[1]{{#1}_{\zar}^{\leq}}

\DeclareMathOperator{\subdiv}{sd}

\renewcommand{\Setfin}{\Set_{\fin}}

\newcommand{\ProSetfin}{\Pro(\Setfin)}

\newcommand{\ProAni}{\Pro(\Ani)}
\newcommand{\Anifin}{\Ani_{\uppi}}
\newcommand{\Anitrun}{\Ani_{<\infty}}
\newcommand{\ProAnifin}{\Pro(\Anifin)}
\newcommand{\ProAnitrun}{\Pro(\Anitrun)}

\renewcommand{\Spec}{\mathrm{Spec}}

\newcommand{\ZZell}{\ZZ_{\el}}

\NewDocumentCommand{\multgrp}{o}{\mathbb{G}\IfValueTF{#1}{_{\mup, #1}}{_{\mup}}}

\newcommand{\proethyp}{_{\proet}^{\hyp}}

\newcommand{\Xproethyp}{X\proethyp}

\newcommand{\Zaraff}[1]{\mathrm{Zar}{}_{#1}^{\kern0.1em\aff}}
\newcommand{\Et}[1]{\mathrm{\acute{E}t}_{#1}}

\newcommand{\ProEt}[1]{\mathrm{Pro\acute{E}t}_{#1}}

\newcommand{\ProEtaff}[1]{\mathrm{Pro\acute{E}t}{}_{#1}^{\kern0.1em\aff}}

\newcommand{\ProZaraff}[1]{\mathrm{ProZar}{}_{#1}^{\kern0.1em\aff}}

\title{\Large On Galois categories and condensed contractible schemes}

\author{Catrin Mair}

\date{\today}

\begin{document}

\maketitle


\begin{abstract} 
 We extend the study of the condensed \textit{Galois category} of a scheme introduced by Barwick, Glasman and Haine in their work on Exodromy \cite{Exodromy}. We elaborate its connection to Lurie's work on Ultracategories \cite{Ultracategories} and provide a description in terms of \textit{$w$-contractible rings}.
 We give a classification of schemes whose Galois category has an initial, respectively, a terminal object. This implies the condensed homotopy type of the scheme, which was studied in more detail in \cite{The_condensed_homotopy_type}, to be trivial.
 Furthermore, we compute a formula for the (underlying group of the) condensed fundamental group of a general Dedekind domain and show that it is non-trivial for $\Spec(\ZZ)$.
 This means that $\Spec(\ZZ)$ is not \textit{condensed contractible}.
\end{abstract}

\tableofcontents

\newpage


\section{Introduction}

Let $X$ be a qcqs scheme. In their work on Exodromy \cite{Exodromy}, Barwick, Glasman and Haine introduce the \textit{Galois category of $X$} $$\Gal(X)\in \Cond(\Cat)\period$$
This condensed category is one variant of a condensed version of the \textit{category of points} $\Pt(X_{\et})$ of the \'etale topos. 
Its name is justified by the fact that it is, somewhat, a globalization of the absolute Galois groups of the residue fields of points of $X$.
The \textit{condensed classifying anima} of $\Gal(X)$, which is given by formally inverting morphisms in the condensed category, defines a \textit{condensed anima} storing homotopy theoretical information on the scheme:
Based on the work on Exodromy, together with Haine, Holzschuh, Lara, Martini and Wolf, we study in \cite{The_condensed_homotopy_type} the \textit{condensed homotopy type} of $X$$$\CondShape{X} \in \CondAni$$ and its \textit{condensed fundamental groups} $\pionecond(X, \xbar)$, for $\xbar \rightarrow X$ a geometric point.
This object constitutes a joint refinement of both the \'etale homotopy type of Friedlander-Artin-Mazur \cites[§4]{MR676809}[§9]{MR245577} and the pro\'etale fundamental group of Bhatt-Scholze \cite[§7]{MR3379634}.
In the joint project, it is shown to be of homotopy theoretical interest beyond that:
On the one hand, it behaves similarly to the \'etale homotopy type and, for example, satisfies a fundamental fiber sequence \cite[Corollary 5.6]{The_condensed_homotopy_type}.
On the other hand, opposite to the classical case, the condensed fundamental group of the affine line over the complex numbers as well as its abelianization are non-trivial \cite[Corollary 7.8]{The_condensed_homotopy_type}
\begin{equation*}
    \picond_1(\AA_{\CC}^1,\xbar)^{(ab)}\neq 1 \period
\end{equation*}  
The non-triviality of the fundamental group can be fixed by passing to the \textit{quasiseparated quotient} for which also a van Kampen formula as well as a Künneth formula hold \cite[Theorem 1.12 and Theorem 1.13]{The_condensed_homotopy_type}.
\subsection*{Motivation}
The result on the condensed fundamental group  $\picond_1(\AA_{\CC}^1,\xbar)$ of the affine line to be non-trivial shows that the condensed homotopy type can store additional information on a scheme compared to its \'etale predecessor: While $\AA_{\CC}^1$ is \'etale contractible it is not \textit{condensed contractible}. This can be stated in terms of local systems. Namely, in \cite[Proposition A.1]{MR4609461} Hemo-Richarz-Scholbach show that the condensed homotopy type $\CondShape{X}$ classifies local systems on $X$ with coefficients in any condensed ring.
From this perspective, the fact that $\AA_{\CC}^1$ is not condensed contractible means that there exists a condensed ring with a non-trivial local system on $\AA_{\CC}^1$. It can be constructed as in \cite[Example 7.10]{The_condensed_homotopy_type}.
An analogous statement does not hold for the \'etale homotopy type \cite[Example 7.4.9]{MR3379634}. The \'etale homotopy type is only known to classify local systems with values in profinite rings such as $\ZZell$.\\
Motivated by this, in this article we focus on the question of when a scheme is \textit{condensed contractible} and give an overview on (non-)examples.
In particular, we compute the group $\picond_1(\Spec(A),\xbar)(\ast)$ for $A$ a general Dedekind domain (\cref{prop:Dedekind_domain}). This provides us with $\Spec(\ZZ)$ as a second example of an \'etale but not condensed contractible scheme, \cref{ex:fundamental_group_integers}.
\subsection*{Main results}
An approach to prove properties of or to compute the condensed homotopy type $\CondShape{X}$ of a qcqs scheme $X$ is to return to $\Gal(X)$ and to draw conclusions from there.
We follow this strategy to examine cases of schemes whose condensed homotopy type is contractible.
\begin{recollection}[(\cref{subsec:condensed_homotopy_type})]
Let $X$ be a qcqs scheme and $\Gal(X)\in \Cond(\Cat)$ its condensed Galois category (\cref{def:Gal}). The \textit{condensed homotopy type} of $X$ is defined as the \textit{condensed classifying anima} (\cref{def:cond_classifiying_anima}) of $\Gal(X)$
$$\CondShape{X}\colonequals\Bcond\Gal(X)\in \CondAni\period$$ The assignment $X\mapsto \CondShape{X}$ defines a functor $\Sch_{\mathrm{qcqs}}\rightarrow \CondAni$ on qcqs schemes.
Fixing a geometric point $\xbar\rightarrow X$ determines a point $\xbar\colon\ast\rightarrow \CondShape{X}$. We have \textit{condensed homotopy group} functors on pointed schemes
    $$\pi_n^{\cond}\colon \Sch_{\mathrm{qcqs},
*}\rightarrow \Cond(\Ccal),\ (\xbar \rightarrow X) \mapsto \pi_n^{\cond}(\CondShape{X},\xbar)\comma $$
where $\Ccal=\Grp$ for $n=1$ and $\Ccal=\Ab\Grp$ for $n\geq 2$,
and a \textit{condensed connected components} functor
$$\pi_0^{\cond}\colon \Sch_{\mathrm{qcqs}}\rightarrow \CondSet, \ X \mapsto \pi_0(\CondShape{X})\period$$
\end{recollection}
\begin{definition}[(\cref{def:condensed_contractible})]
    Let $X$ be a qcqs scheme.
    \begin{enumerate}
        \item We call $X$ \textit{condensed connected} if its condensed set of connected components is $\pizerocond(X)=\ast$. 
        \item We call $X$ is \textit{condensed contractible} if its condensed homotopy type is $\CondShape{X}=\ast$. 
    \end{enumerate}
\end{definition}
\begin{remark}[(\cref{rem:trivial_on_fundgroups})]
Equivalently, a scheme is condensed contractible if it is condensed connected and all its condensed homotopy groups are trivial.    
\end{remark}
In \cref{sec:via_weakly_proobjects}, we prove an alternative description of $\Gal(X)$ using the notion of \textit{$w$-contractible rings}  introduced by Bhatt-Scholze in \cite{MR3379634}. 
\begin{theorem}{(\cref{prop:Galois_weaklycontractible})}
Let $X$ be an affine scheme. For every $S \in \Extr$, the category $\Gal(X)(S)$ identifies with the category whose objects are $w$-contractible rings $A$ lying above $X$, i.e., coming with a structure map $\Spec(A) \rightarrow X$, satisfying $\pi_0(\Spec(A))=S$ and whose morphisms are given by maps of rings inducing the identity on $\pi_0$ of the induced map of prime spectra.
\end{theorem}
It finds application in \cref{sec:Trivial_Shapes} where we prove a classification of \textit{condensed contractible} schemes.
\begin{theorem}
Let $X$ be a qcqs scheme and $\Pt(X_{\et})$ the category of points of its \'etale topos and $\Gal(X)$ its condensed Galois category. Then we have equivalences:
\begin{itemize}
\item[a)] $X$ is everywhere strictly local and irreducible. $\iff$ $\Gal(X)$ is cofiltered. $\iff$ $\Gal(X)$ has a terminal object.  $\iff$ $\Pt(X_{\et})$ has a terminal object. (\cref{thm:Galterminal})
\item[b)] $X$ is the spectrum of a strictly henselian local ring. $\iff$ $\Gal(X)$ is filtered. $\iff$ $\Gal(X)$ has an initial object. $\iff$ $\Pt(X_{\et})$ has an initial object. (\cref{thm:Galinitial})
\end{itemize}
In particular, in both cases, a) and b), the schemes are \textit{condensed contractible}.
\end{theorem}
Beyond this, we provide further (non-)examples of \textit{condensed contractible} schemes in \cref{sec:condensed_contractible}. Of most interest is the example on the spectrum of the integers.
\begin{theorem}{(\cref{prop:Dedekind_domain})}\label[theorem]{thm:condensedfundamental_Dedekind}
Let $A$ be a Dedekind domain with fraction field $K=\Frac(A)$ and generic point $\etabar \rightarrow \Spec(A)$. The \textit{underlying group} of the condensed fundamental group is the quotient $$\pionecond(\Spec(A), \etabar)(\ast)=G_{K}/N\comma $$ where $G_{K}$ is the absolute Galois group of $K$ and $N$ the (abstract) normal closure of the subgroup generated by all inertia groups $I_p$ at finite primes $p\subset A$.
\end{theorem}
\begin{example}[(\cref{ex:fundamental_group_integers})]
For $A=\ZZ$, the quotient in \cref{thm:condensedfundamental_Dedekind} and its abelianization are non-trivial. Thus the whole condensed fundamental group and its abelianization are non-trivial $$\picond_1(\Spec(\ZZ),\etabar)^{(\ab)}\neq 1\period $$
In other words, $\Spec(\ZZ)$ is not condensed simply connected and thus not condensed contractible.
However, the scheme is known to have trivial (extended) \'etale homotopy groups \cite[Example 11.2.2]{Exodromy} and, by normality, also a trivial pro-\'etale fundamental group \cite[Lemma 7.4.10]{MR3379634}. 
\end{example}
It follows from the example that there exists a non-trivial local system of $\Spec(\ZZ)$ with values in a condensed ring, constructed accordingly to \cite[Example 7.10]{The_condensed_homotopy_type}.


\subsection*{Conventions}\label{intro_subsec:notational_conventions}
\subsubsection*{Set-theoretic conventions}
Whenever working with condensed mathematics, there arise set-theoretic issues from the fact that defining sites are large. We follow the approach used in the setting of \textit{pyknotic mathematics} \cites{pyknoticI}{exodromy}  to deal with these. More on this can be found in \cref{ignorsetiss}. As an upshot: We will neglect set-theoretic issues in the sense that we will implicitly work with $\kappa$-small objects or functors for strongly inaccessible cardinals $\kappa$. We use results on condensed/pyknotic mathematics from \cites{pyknoticI}{Exodromy}{Scholze:condensednotes}{Scholze:analyticnotes} without any further notice on set-theory. 
\subsubsection*{On higher category theory}
We follow the conventions by Lurie in \cite{HTT}, \cite{Kerodon} and \cite{SAG} regarding aspects of higher category theory like $\infty$-categorical definitions and constructions.
The prefix "$\infty$" might be omitted if it is clear that all objects are considered in the $\infty$-categorical setting.
Moreover, we make the following implicit assumptions:
\begin{itemize}
\item Whenever working with a $1$-category in the $\infty$-setting, it will be implicitly identified with its corresponding $\infty$-category by the nerve construction.
\item If there exists a fully faithful embedding of $\infty$-categories 
$$\iota\colon \Ccal \hookrightarrow \mathcal{D}\comma $$ we will usually identify an object $X \in \Ccal$ with its image under $\iota$ and will refer to the corresponding object in $\mathcal{D}$ with $X$ as well.
\item For an $\infty$-category $\Ccal$, we denote the $\infty$-category of pro-objects on $\Ccal$, which is naturally identified with the opposite of the \category of left exact accessible functors $ \fromto{\Ccal}{\Ani} $, by
\begin{equation*}
        \Pro(\Ccal) \equivalent \Funlexacc(\Ccal,\Ani)^{\op} \comma
\end{equation*}
see \cite[\SAGthm{Definition}{A.8.1.1} \& \SAGthm{Proposition}{A.8.1.6}]{SAG}. We will implicitly identify an object in $\Ccal$ with its image under $\Ccal\hookrightarrow\Pro(\Ccal)$. 
\end{itemize}
\subsubsection*{Notational conventions}
Throughout this paper, there appear several ($\infty$-)categories. 
We fix the following (standard) notation.
\begin{enumerate}
    \item We write $\Set$ for the category of sets and $\Set^{\mathrm{fin}}$ for the category of finite sets.
    \item We write $\Grp$ for the category of groups and $\Ab\Grp$ for the category of abelian groups.
     \item  We write $\Top$ for the category of topological spaces, $\Comp$ for the category of compact Hausdorff spaces, $\ProFin$ for the category of totally disconnected compact Hausdorff spaces (=profinite sets) and further $\Extr$ for the category of extremally disconnected profinite sets.
	\item We write $\Cat$, respectively, $ \Catinfty $ for the large ($\infty$-)category of small ($\infty$-)categories, and use $ \Ani \subset \Catinfty $ for the full subcategory spanned by the anima (\groupoids/ spaces). By $\Ani_{< \infty}$ and $\Ani_{\pi}$ we denote the subcategories of truncated, respectively, of $\pi$-finite anima.
    \item We write $\Cond(\Ccal)$ for the ($\infty$-)category of condensed objects on some ($\infty$-)category $\Ccal$.
    \item We write $\Fun(\Ccal, \mathcal{D})$ for the $\infty$-category of functors between $\infty$-categories $\Ccal$ and $\mathcal{D}$ and $\Functs(\Ccal, \mathcal{D})$ for the $\infty$-category of functors between condensed $\infty$-categories $\Ccal$ and $\mathcal{D}$.
    \item We write $\RTop_{\infty}$ for the $\infty$-category of $\infty$-topoi with \textit{geometric morphisms}, i.e., a morphism $f\colon \mathcal{X} \rightarrow \mathcal{Y}$ between two $\infty$-topoi is given by a functor $f_*\colon \mathcal{X} \rightarrow \mathcal{Y} $  admitting a left exact left adjoint $f^* \colon \mathcal{X} \rightarrow \mathcal{Y}$. We refer to the left adjoint $f^*$ as an \textit{algebraic morphism}. The $\infty$-category of morphisms between two $\infty$-topoi $\mathcal{X}$ and $\mathcal{Y}$ in $\RTop_{\infty}$ is denoted by $\Fun_*(\mathcal{X},\mathcal{Y})$. It is the opposite category of the $\infty$-category of algebraic morphisms denoted by $\Fun^*(\mathcal{X},\mathcal{Y})$.
     \item For a site $\Ccal$ (with a Grothendieck topology $\tau$) we write $\Sh(\Ccal)$ for the $\infty$-category of sheaves of anima on $(\Ccal,\tau)$, $ \Shhyp(\Ccal) \subset \Sh(\Ccal)$ for the subcategory of hypercomplete objects (i.e., sheaves of anima satisfying descent with respect to hypercoverings), and $\Sh(\Ccal, \Set)$ for the category of sheaves of sets on $(\Ccal, \tau)$, i.e., the subcategory of $0$-truncated objects in $\Sh(\Ccal)$.
    \item We write $\Sch_{\mathrm{qcqs}}$ for the category of qcqs schemes. Given a scheme $ X $, we write $\Et{X}$ and $\ProEt{X}$ for its \emph{étale} and \emph{proétale site}, respectively.
    Moreover, we write $X_{\et} \colonequals \Sh(\Et{X})$ and $X_{\proet} \colonequals \Sh(\ProEt{X})$ for the \'etale and pro-\'etale \topoi on $ X $, respectively. 
\end{enumerate}


\subsection*{Acknowledgments}\label{intro_subsec:acknowledgments}

This paper emerged from the second chapter of my dissertation \cite{CatrinsThesis}.
I would like to once again thank my advisor, Torsten Wedhorn, for all his support, as well as the referees Clark Barwick and Timo Richarz.
Special thanks go to Peter Scholze for answering my questions communicated via email, Clark Barwick and Peter Haine for explaining me their work on exodromy, and to my coauthors in the joint project \cite{The_condensed_homotopy_type} on the condensed homotopy type: Peter Haine, Tim Holzschuh, Marcin Lara, Louis Martini and Sebastian Wolf.
I also want to thank Jakob Stix for the final argument in \cref{ex:fundamental_group_integers}, as well as Marcin Lara and Alex Youcis for insights on this. 
I gratefully acknowledge support by Deutsche For\-schungs\-ge\-mein\-schaft (DFG, German Research Foundation) through the Collaborative Research Centre TRR 326 Geometry and Arithmetic of Uniformized Structures, project number 444845124, and Germany's Excellence Strategy EXC 2044/2 - 390685587, Mathematics Münster: Dynamics-Geometry-Structure.


\section{Preliminaries}\label[section]{sec:preliminaries}
In this section, we collect basics on categorical aspects of condensed mathematics, recap Galois categories of schemes and give an upshot on condensed homotopy types of schemes before we provide more explicit descriptions, examples and computations on these objects in \cref{sec:via_weakly_proobjects} and \cref{sec:condensed_contractible}. 
\subsection{Recollection on condensed mathematics}
We review  condensed sets, condensed anima and condensed $\infty$-categories. As additional sources consult \cite{pyknoticI} and \cite{Scholze:condensednotes}. Note \cref{ignorsetiss} for a comment on set-theoretic issues that come along with definitions in the world of condensed mathematics.
\subsubsection{Condensed Sets and condensed anima}
The basic objects in condensed mathematics are \textit{condensed sets}. Originally, these are defined as sheaves of sets on the \textit{pro-\'{e}tale site} $\ProEt{*}$ of a point, i.e., the spectrum of a separably closed field. See \cite{MR3379634} for an introduction to the \textit{pro-\'etale topology of schemes.} The category of condensed sets is denoted by $$\CondSet \colonequals \Sh(\ProEt{*},\Set) \period$$
Sheaves on $\ProEt{*}$ can be identified with sheaves on certain subcategories of topological spaces. The notion of condensed sets extends in the $\infty$-categorical world to the notion of \textit{condensed anima}, denoted by $\CondAni$, by considering \textit{hypersheaves} with values in the $\infty$-category $\Ani$ of \textit{anima} (or spaces).
\begin{notation}\label[notation]{not:topo_notions}
We denote by $\ProFin$ the category of \textit{profinite sets}. By Stone duality, the category $\ProFin$ can be identified with the full subcategory of totally disconnected spaces in the category $\Comp$ of \textit{compact Hausdorff spaces}.
By a result by Gleason, stated in \cite[p. III 3.7]{Johnstone1982StoneSpaces}, the projective objects in $\Comp$ are the \textit{extremally disconnected} profinite sets $\Extr \subset \Comp$.
\end{notation}
\begin{recollection}\label[recollection]{rec:definingsites}
There are at least four different defining sites for the $\infty$-category $\CondAni$ of \textit{condensed anima} defining the same $\infty$-topos of hypercomplete $\infty$-sheaves:
\begin{itemize}
    \item[(i)] The pro-\'{e}tale site $\ProEt{\Spec(k)}$ for $k$ any separably closed field,
    \item[(ii)] the site $\Comp$ of compact Hausdorff spaces,
    \item[(iii)] the site $\ProFin$ of profinite sets, 
    \item[(iv)] the site $\Extr$ of extremally disconnected profinite sets.
\end{itemize}
Here, coverings for the sites (ii)-(iv) are given by jointly surjective families of continuous maps of topological spaces.
All of these sites define the same $\infty$-topos on hypercomplete $\infty$-sheaves: Every compact Hausdorff space admits a surjection from an extremally disconnected compact Hausdorff space, namely from the Čech--Stone compactification of its underlying discrete space. Moreover, $\ProFin$ identifies with the subcategory of affine schemes in $\ProEt{\Spec(k)}$. Thus there are hierarchies of bases 
$ \Extr \subset \ProFin \subset \Comp \text{ and } \ProFin \subset \ProEt{\Spec(k)}.$
By \cite[Corollary A.7]{arXiv:2001.00319}, restriction of sites induces equivalences of hypercomplete \topoi
    \begin{equation}\label{eq:equivalent_descriptions_of_condensed_anima}
        \Shhyp(\ProEt{\Spec(k)}) \equivalence \Shhyp(\ProSetfin) \equivalence \Shhyp(\Extr) \period
    \end{equation}
    The \topos $\CondAni$ of \textit{condensed anima} is any of the equivalent \topoi \eqref{eq:equivalent_descriptions_of_condensed_anima}.
\end{recollection}
\begin{recollection}\label[recollection]{rec:fully_faithful_condset_to_condani}
    The category $\CondSet$ of condensed sets embeds fully faithfully into the $\infty$-category $\CondAni$ of condensed anima by post-composition of functors with the fully faithful embedding $v\colon\Set \hookrightarrow \Ani$. The left adjoint is given by post-composition with the $0$-truncation functor $\pi_0\colon \Ani \rightarrow \Set$. In other words, the objects of $\CondAni$ in the essential image of $\CondSet$ are exactly the $0$-truncated objects. They are called \textit{static} condensed anima.
\end{recollection}
\begin{remark}\label[remark]{rem:extdis_projective}
The sites $\Comp$ and $\Extr$ in \cref{rec:definingsites} can be described in more topos theoretic means:
\begin{itemize}
    \item[a)] The site $\Comp$ is the subcategory of $0$-truncated \textit{coherent} objects of $\CondAni$, see \cite[Theorem 2.16]{Scholze:condensednotes}.
    \item[b)] The site $\Extr$ is the subcategory of (0-truncated) \textit{compact projective} objects of $\CondAni$, compare \cites[Chapter 3]{arXiv:2012.10502}[Example 11.3 (4)]{Scholze:analyticnotes}.
    Equivalently, by a) and \cref{not:topo_notions} they form the subcategory of $0$-truncated coherent \textit{projective} objects in $\CondAni$.
\end{itemize}
\end{remark}
\begin{recollection}\label[recollection]{rec:fin_prod_pres_presheaves}
On extremally disconnected profinite sets, hypersheafification can be omitted and the $\infty$-sheaf condition simplifies. More precisely, a presheaf $ F $ on $ \Extr $ is a hypersheaf if and only if $ F $ carries finite disjoint unions to finite products, i.e., there is an identification
    \begin{equation*}
       \CondAni \simeq \Shhyp(\Extr) \equivalent \Funcross(\Extrop,\Ani)\comma
    \end{equation*}
    where the right hand site denotes finite product-preserving functors. It is a consequence of $\Extr$ being a subsite of projective objects and widely used in the literature on condensed mathematics. A detailed discussion can be, for example, found in \cite[Section A.3.5]{CatrinsThesis}.
\end{recollection}
\begin{recollection}
The \textit{animation} of an ordinary category is defined as the $\infty$-category freely generated under sifted colimits by compact projective objects, see \cite[Section 11.1]{Scholze:analyticnotes}. For example, the \category $\Ani$ of anima (or animated sets) is the animation of the category $\Set$ of sets where the compact projectives are exactly the finite sets, see \cite[Example 11.5]{Scholze:analyticnotes}.
\end{recollection}
\begin{definition}{\textbf{Animation of condensed sets}}\label{anicond}\\
The animation $\Ani(\CondSet)$ of the category $\CondSet$ is the $\infty$-category freely generated under sifted colimits by the category of extremally disconnected profinite sets $\Extr$, see \cref{rem:extdis_projective}. 
Equivalently, it can be defined as the full $\infty$-subcategory of functors in
$$\mathrm{Fun}(\Extr^{\mathrm{op}}, \Ani)$$ generated under sifted colimits by the Yoneda image.
\end{definition}
\begin{remark}
From the description of $\CondAni$ in \cref{rec:fin_prod_pres_presheaves}, it follows that sifted colimits in $\CondAni$ can be computed in the presheaf category $ \Fun(\Extr^{\op},\Ani) $ and that there is an equivalence of $\infty$-categories
$$\CondAni \simeq \Ani(\CondSet)\comma $$
also see \cite[Definition 11.7. and Lemma 11.8]{Scholze:analyticnotes}.
\end{remark}
Condensed sets can be seen as a replacement of topological spaces in the following sense.
\begin{example}\label[example]{ex:top}
To a topological space $T\in \Top$, we associate a condensed set $\underline{T}\in \CondSet$ via the restricted Yoneda embedding
\begin{align*}
\underline{T}\colon\Extr^{\mathrm{op}} &\rightarrow \Set, \
S \mapsto \Hom_{\Top}(S,T)\period
\end{align*}
If $T$ is a topological group, every set $\Hom_{\Top}(S,T)$ inherits the group structure and $\underline{T}$ is a group object in $\CondSet$.
Speaking of a topological space in the context of condensed sets, we will usually mean the condensed set associated to the topological space under this functor.\footnote{Note that there are set-theoretical issues in the definition of condensed objects, see \cref{ignorsetiss}. For that reason this construction will not give a condensed set in the sense of Clausen-Scholze \cite{Scholze:condensednotes} if the space $T$ is not $T_1$. This is also commented in \cite[Section 0.3]{pyknoticI} and \cite[Warning 2.14., Proposition 2.15]{Scholze:condensednotes}.}
\end{example}
\begin{remark}
The functor $\Top \rightarrow \CondSet$ is not fully faithful in general but when restricted to the full subcategory of $ \Top $ spanned by the compactly generated topological spaces, see \cite[Proposition 1.7]{Scholze:condensednotes}.
\end{remark}
The $\infty$-category $\CondAni$ combines the topological space direction of $\CondSet$ with the homotopy theoretical direction of $\Ani$.
\begin{recollection}\label[recollection]{rec:discrete_condani}
The $\infty$-category $\Ani$ embeds fully faithfully into $\CondAni$ by the pullback morphism $(\cdot)^{\mathrm{disc}}\colon \Ani \hookrightarrow \CondAni$ of the unique geometric morphism to the (terminal) $\infty$-topos $\Ani$. 
We refer to condensed anima in the essential image as \textit{discrete} condensed anima.
Its right adjoint is the global sections functor $\mathrm{ev}_*\colon \CondAni \rightarrow \Ani$ given by $A\mapsto A(*)$.
\end{recollection}
\subsubsection{Condensed objects in \texorpdfstring{$\infty$}{infinity}-categories}
The description of condensed anima in \cref{rec:fin_prod_pres_presheaves} extends to a general definition of condensed objects of some \category $\Ccal$.
\begin{definition}\label{def:condobj}
Let $\Ccal$ be an $\infty$-category that admits all finite products. We denote the $\infty$-category of finite product-preserving presheaves $\fromto{\Extr^{\op}}{\Ccal}$ by
$$\Cond(\Ccal)\colonequals \Fun^{\times}(\Extr^{\op}, \Ccal)$$ and refer to its objects as \textit{condensed objects of $\Ccal$}.
\end{definition}
We mainly apply this definition to the ($\infty$-)categories $\Ccal=\Set,\Grp, \Ani$, $\Cat$ and $\Catinfty$.
\begin{warning}\label[warning]{ignorsetiss}
The definition of condensed objects comes with set-theoretic problems due to the size of the category $\Extr$ of extremally disconnected profinite sets. The considered functor categories are not locally small. As we highly rely on the work of Barwick-Glasman-Haine in \cite{Exodromy}, we follow the suggestion of \cite{pyknoticI} to handle these issues: 
We implicitly fix two strongly inaccessible cardinals $ \delta < \epsilon $ and think about $\Cond(\Ccal)$ as hypersheaves of $\epsilon$-small objects in $\Ccal$ on $\delta$-small objects in $\Extr$. However, as these choices do not affect our results, we will not further comment on set-theoretic issues.
\end{warning}
We have a general version of \cref{rec:discrete_condani} for condensed objects of some $\infty$-category $\Ccal$.
\begin{recollection}
For every $\infty$-category $\Ccal$ such that $\Cond(\Ccal)$ is defined, there is an adjoint pair of functors 
\[
 \begin{tikzcd}[column sep = huge]
           \Cond(\Ccal) \arrow[r, shift left=1ex] & \Ccal \arrow[l, hook',shift left=.5ex, swap]
           \end{tikzcd}
   \]
given by the constant sheaf and global sections functors.
For an object $X\in \Ccal$, the image under $\Ccal \hookrightarrow \Cond(\Ccal)$ is the \textit{discrete condensed object} $X^{\mathrm{disc}}$ given by the assignment
 $$S=\{S_i\}_{i\in I}\mapsto X^{\mathrm{disc}}(S)\colonequals \colim_{i \in I} X^{S_i}\comma $$
 where $X^{S_i}$ denotes the product $\prod_{S_i}X$.
 The right adjoint $\Cond(\Ccal) \rightarrow \Ccal, \ X \mapsto X(*)$ sends a condensed object of $\Ccal$ to its \textit{underlying object} in $\Ccal$.
\end{recollection}
The following generalizes the adjoint pair of functors of \cref{rec:fully_faithful_condset_to_condani}.
\begin{recollection}\label[recollection]{rec:adjointcondensed}
If $ \Dcal $ is another \category with finite products and $ F \colon \fromto{\Ccal}{\Dcal} $ is a finite product-preserving functor, we write
    \begin{equation*}
        F^{\cond} \colon \fromto{\Cond(\Ccal)}{\Cond(\Dcal)}
    \end{equation*}
    for the functor given by post-composition with $ F $.
    If $ F \colon \fromto{\Ccal}{\Dcal} $ admits a right adjoint $ G $, then $ G^{\cond} $ is right adjoint to $ F^{\cond} $.
\end{recollection}
\begin{example}\label[example]{ex:condensedhomogrps}
There are \textit{condensed homotopy group functors} on pointed condensed anima
$$\pi_n^{\cond}\colon \CondAni_*\rightarrow \CondGrp$$
defined by level-wise post-composition with the simplicial homotopy group functors on pointed anima
$\pi_n\colon\Ani_*\rightarrow \Grp,$ for all $n\geq 1$.
\end{example}
For $\Ccal=\Catinfty$, we obtain the $\infty$-category $\CondCat$ of \textit{condensed $\infty$-categories}.
\begin{example}
	We can define condensed \categories $ \ICond(\Ani) $ and $ \ICond(\Set) $ by the assignments
	\begin{equation*}
		S \mapsto \CondAni_{/S} \andeq S \mapsto \Cond(\Set)_{/S} \period
	\end{equation*}
\end{example}
Functors of condensed $\infty$-categories are given as follows.
\begin{recollection}\label{rec:contfunc}
Let $\Ccal$ and $\mathcal{D}$ be condensed $\infty$-categories. 
The $\infty$-category $\Functs(\Ccal,\mathcal{D})$ of \textit{continuous functors between condensed $\infty$-categories} is defined as the end of the bifunctor $$(\Extr^{\mathrm{op}})^{\mathrm{op}} \times \Extr^{\mathrm{op}} \rightarrow \Ccal, \ S\times S' \mapsto \Fun(\Ccal(S), \mathcal{D}(S'))\comma $$ 
compare \cite[Definition 13.3.16]{Exodromy}. As proven in \cite[Proposition 2.3]{MR3518559}, this is equivalent to the $\infty$-category $\mathrm{Nat}(\Ccal,\mathcal{D})$ of natural transformations of condensed $\infty$-categories.
We can think of the objects  of $\Functs(\Ccal,\mathcal{D})$ as tuples of functors $\Ccal(S)\rightarrow \mathcal{D}(S)$ indexed by $S\in \Extr$ satisfying certain compatibility conditions. 
\end{recollection}
\subsubsection{The condensed classifying anima}
The inclusion of $\infty$-groupoids into $\infty$-categories
$\Ani \hookrightarrow \Catinfty$ admits both a left adjoint and a right adjoint. 
The left adjoint formally inverts all morphisms (in an $\infty$-categorical sense).
\begin{notation}
	We denote the left adjoint to the inclusion $ \Ani \inclusion \Catinfty $ by $ \Bup \colon \fromto{\Catinfty}{\Ani} $.
	Given \acategory $ \Ccal $, we call $ \Bup\Ccal $ the \defn{classifying anima} of $ \Ccal $.
\end{notation}
As $\Bup\colon \Catinfty \rightarrow \Ani$ preserves finite products, under notice of \cref{rec:adjointcondensed} there exists a corresponding condensed version of this functor.
\begin{definition}\label[definition]{def:cond_classifiying_anima}
We define the \textit{condensed classifying anima} functor 
$$\Bcond\colon \CondCat \rightarrow \CondAni$$
by level-wise post-composition with the classifying anima functor 
$\Bup\colon \Catinfty \rightarrow \Ani.$
More explicitly, the \textit{condensed classifying anima} $$\Bcond\Ccal\in \CondAni$$ of a condensed $\infty$-category $\Ccal$ sends any $S\in \Extr$ to the classifying anima $\Bup\Ccal(S) \in \Ani$ of the $\infty$-category $\Ccal(S) \in \Catinfty$.
\end{definition}
\begin{remark}
The condensed classifying anima functor is the left adjoint to the inclusion $\CondAni \hookrightarrow \CondCat,$ which is defined by post-composition with $\Ani \hookrightarrow \Catinfty$.   
\end{remark}


\subsection{The Galois category of a scheme}\label[subsection]{sec:Comparison_Approaches}
We recall the presentation of the \textit{Galois category of a scheme}, which was introduced in \cite{Exodromy} as a profinite (layered) $\infty$-category. We give its definition as a condensed category. The construction is meaningful in the general setting of coherent \topoi, see also \cite[Chapter 3]{The_condensed_homotopy_type}. 

\begin{recollection}\label{rec:coherent_is_constructible}
    Let $ X $ be a qcqs scheme.
    Then the \'etale \topos $ X_{\et} $ is coherent \cite[Proposition 3.7.3]{Exodromy} and by \cite[Lemma 9.5.3 \& Proposition 9.5.4]{Exodromy}, the truncated coherent objects of $ X_{\et} $ are exactly the constructible étale sheaves of anima on $ X $.
\end{recollection}

\begin{definition}\label[definition]{def:Gal}
Let $X$ be a qcqs scheme.
The \defn{Galois \category} of $ X$ is the condensed \category$ \Gal(X) $ defined by the functor
    \begin{align*}
        \Extr^{\op} &\to \Catinfty, \ 
        S \mapsto \Fun^{*,\coh}(X_{\et},\Sh(S)) \period
    \end{align*}
    Here, $ \Fun^{*,\coh}(X_{\et},\Sh(S)) $ is the \category of \defn{coherent} algebraic morphisms $ \supperstar \colon X_{\et} \to \Sh(S) $ of \topoi. These are those algebraic morphisms that send constructible \'etale sheaves of anima on $ X $ to locally constant constructible sheaves of anima on the topological space $ S $. 
\end{definition}
\begin{remark}
    Originally, the Galois category is defined as the stratified homotopy type of the \'etale topos stratified over the Zariski space of the scheme \cite[Definition 12.1.3]{Exodromy}. This description will be useful later in \cref{prop:Dedekind_domain}.
    Our description appears roughly in \cite[Example 13.5.4]{Exodromy}.
    There pushforward functors are considered and the coherent assumption is missing. This is corrected in \cite[Definition 3.24]{The_condensed_homotopy_type}.
\end{remark}
\begin{remark}
    Let $X$ be a qcqs scheme.
    As the \topos $X_{\et}$ is $1$-localic (i.e., the underlying \'etale site is an $1$-category), by \cite[Lemma 6.4.5.6]{HTT} we can identify $$\Gal(X)(S)=\Fun^{*,\coh}(X_{\et},\Sh(S))=\Fun^{*,\coh}(\Sh(X_{\et},\Set),\Sh(S,\Set))\in \Cat.$$ We can thus work with coherent algebraic morphisms on the level of $1$-topoi. These send constructible \'etale sheaves of sets to locally constant constant constructible sheaves of sets on $S$. 
\end{remark}
Referring to the previous remark, we can view $\Gal(X)$ as a condensed $1$-category.
The underlying category of $\Gal(X)$, i.e., the global sections $ \Gal(X)(\pt)\in \Cat $, recovers the category of \textit{points} $ \Pt(X_{\et}) $ of the étale topos of $ X $. We recall the definition of a point of an $\infty$-topos.
\begin{definition}
A \textit{point} of an $\infty$-topos $\mathcal{X}\in \RTop_{\infty}$ is a geometric morphism $p_*\colon \Ani \rightarrow \mathcal{X}$.
The \textit{$\infty$-category of points} of $\mathcal{X}$ is the opposite category of such functors $$\Pt(\mathcal{X})\colonequals \Fun_*(\Ani, \mathcal{X})^{\op}=\Fun^*(\mathcal{X}, \Ani)\period$$
\end{definition}
\begin{remark}
Let $X$ be a qcqs scheme. Due to $1$-locality of the \'etale $\infty$-topos, the $\infty$-category $\Pt(X_{\et})$ of points of the \'etale $\infty$-topos $X_{\et}$ is a $1$-category.    
\end{remark}
\begin{lemma}
Let $X$ be a qcqs scheme.
There is an equivalence of categories
$$\Pt(X_{\et})\simeq\Gal(X)(*)\period$$
\end{lemma}
\begin{proof}
We need to show that all algebraic morphisms from $X_{\et}$ to $\Ani$ are automatically coherent.
This follows from \cite[Proposition 9.5.7]{Exodromy}, as $\Ani$ and $X_{\et}$ are spectral $\infty$-topoi and every morphism $*\rightarrow |X|$ to the underlying Zariski space $|X|$ of $X$ is quasi-compact.
\end{proof}
For every qcqs scheme $X$, there is an explicit description of the category of points $\Pt(X_{\et})$ in terms of geometric points of $X$ and their lifts, which was already presented by the Grothendieck school in \cite[Expos\'e VIII, Th\'eor\`eme 7.9]{SGA4ii}, also see \cite[ Construction 0.0.1., Construction 0.0.2., Construction 12.1.5]{Exodromy}:
\begin{recollection}{\textbf{Points of the \'{e}tale topos}}\label[remark]{rem:Pointsetaletopos}\\
Let $X$ be a qcqs scheme.
The category of points $\Pt(X_{\et})$ of the \'{e}tale ($\infty$-)topos of $X$ can be characterized as follows: 
\begin{itemize}
\item \textbf{Objects are given by geometric points:}\\
A geometric point $\xbar \rightarrow X$ corresponds to a scheme theoretic point $\Spec(\kappa(x))\rightarrow X$ where $\kappa(x)$ is a separable closure of the residue field $\kappa(x_0)$ of its image $x_0 \in |X|$.
\item \textbf{Morphisms are given by lifts of geometric points to strict localizations:}\\
For $\xbar\rightarrow X$ and $\ybar\rightarrow X$ geometric points, a morphism $\xbar\rightarrow \ybar$ is defined as a morphism $\ybar\rightarrow X_{(x)}\colonequals \Spec(\mathcal{O}_{X,x}^{\mathrm{sh}})$ such that the triangle commutes:
\[
  \begin{tikzcd}
   \Spec(\kappa(y)) \arrow[dotted]{r} \arrow[swap]{dr} & \Spec(\mathcal{O}_{X,x}^{\mathrm{sh}}) \arrow{d} \\
     & X
  \end{tikzcd}
\]
\end{itemize}
\end{recollection}
Here, the prime spectrum of the strict henselization $\Spec(\mathcal{O}_{X,x}^{\mathrm{sh}})$ of the local ring $\mathcal{O}_{X,x}$ is also called \textit{strict localization} of $X$ at the geometric point $\xbar$.
\begin{recollection}{\textbf{Morphisms of points}}\label[remark]{rem:Pointsetaletopos2}\\
 Let $\xbar\rightarrow X$ and $\ybar\rightarrow X$ be geometric points corresponding to points $x_0$ and $y_0$ in $|X|$ and $\Hom_{\Pt(X_{\et})}(\xbar,\ybar)$ the set of morphisms between $\xbar$ and $\ybar$ in $\Pt(X_{\et})$.
\begin{itemize}
\item[1.] The set of morphisms $\Hom_{\Pt(X_{\et})}(\xbar,\ybar)$ is non-empty if and only if $x_0$ is a specialization of $y_0$ in $|X|$, i.e., $x_0$ is in the closure of $y_0$.
\item[2.] The morphism set $\Hom_{\Pt(X_{\et})}(\xbar,\xbar)$ of a geometric point is the absolute Galois group of $\kappa(x_0)$ defined as the Galois group $\Gal(\kappa(x)/\kappa(x_0))$ of the separable closure $\kappa(x)$.
\end{itemize}
\end{recollection}
\begin{remark}\label[remark]{rem:points_strictlyhenselian}
The category $\Pt(X_{\et})$ equivalently is the category of strictly henselian local rings lying above $X$. This follows by identifying a geometric point $\xbar\rightarrow X$ with the strict localization $\Spec(\mathcal{O}_{X,x}^{\mathrm{sh}})$ of $X$ at the geometric point $\xbar$. This result is classical, for a reference see, e.g., \cite{GABBER20154667}.
\end{remark}
We prove a generalization of this remark for $\Gal(X)$ in \cref{sec:via_weakly_proobjects}.
\begin{recollection}\label[recollection]{rec:Zariskiposet}
The underlying Zariski space $|X|$ of a scheme $X$ carries the structure of a poset via the partial order induced by specializations.
For two points $x, y$ of a topological space we say $x\leq y$ if and only if $x$ is a specialization of  $y$, i.e., $x$ lies in the closure of $y$. We denote the \textit{Zariski (or specialization) poset} of the underlying Zariski space $|X|$ of a scheme $X$ by $ \Zpos{X} $.
\end{recollection}
\begin{example}\label[example]{ex:Gal(X)_is_a_stratified_anima}
    For every qcqs scheme, there is a natural functor
    \begin{equation*}
        s \colon \Gal(X)(\pt) \to \Zpos{X}
    \end{equation*}
    from the category of points of the étale topos to the specialization poset of $ |X| $.
    It sends a geometric point $ \xbar \to X $ to the underlying point $ x_0 \in |X| $ and a morphism $\xbar\rightarrow \ybar$ of geometric points to the order $x_0\leq y_0$ as every morphism in $\Gal(X)(*)=\Pt(X_{\et})$ corresponds to a specialization.

\end{example}
\subsection{The condensed homotopy type of a scheme}\label[section]{subsec:condensed_homotopy_type}
The condensed classifying anima of the Galois category defines the \textit{condensed homotopy type} of a scheme. This object was first mentioned in \cite{Exodromy} and studied in more detail in \cite{The_condensed_homotopy_type}.
\begin{definition}
    Let $X$ be a qcqs scheme. We refer to the condensed classifying anima 
    $$\CondShape{X}\colonequals\Bcond\Gal(X)\in \CondAni$$ of the Galois category of $X$ as the \textit{condensed homotopy type of $X$.}
\end{definition}
\begin{recollection}
The assignment $X\mapsto \CondShape{X}$ defines a functor 
\begin{align}\label{fun:condensed_shape}
\Shape^{\cond} \colon \Sch_{\mathrm{qcqs}}\rightarrow \CondAni
\end{align}
from the category of qcqs schemes to the $\infty$-category of condensed anima. 
\end{recollection}
\begin{recollection}\label[recollection]{rec:recover_etale_homotopy}
The condensed homotopy type of a qcqs scheme $X$ can be seen as a condensed refinement of the \'etale homotopy type $\Piet(X)\in \ProAni$ of $X$. To be more precise, the \'etale homotopy type $\Piet(X)$ is, up to protruncation, the image of $\CondShape{X}\in \CondAni$ under the \textit{prodiscretization} functor $(-)\prodisccompl\colon\CondAni \rightarrow \ProAnitrun$, see \cite[Lemma 3.14]{The_condensed_homotopy_type}.
\end{recollection}
\begin{remark}\label[remark]{rem:relative_shape}
Similar as $\Piet(X)$ can be defined as the \textit{shape} of the \'etale $\infty$-topos relative to $\Ani$, see \cite[Corollary 5.6]{HoyoisGalois}, the condensed homotopy type finds a description as the shape of the (hypercomplete) pro-\'etale $\infty$-topos $\Xproethyp\colonequals\Shhyp(\ProEt{X})$ of $X$ relative to $\CondAni$.    
We omit details on this perspective here as our studies totally rely on the definition of $\CondShape{X}$ via $\Gal(X)$. For details, consult \cite[Sections 2.2.1 and 2.2.3]{CatrinsThesis} or \cite[Section 3.1, Proposition 3.38]{The_condensed_homotopy_type}.
\end{remark}
\begin{recollection}
Fixing a geometric point $\xbar\rightarrow X$ determines a point $\xbar\colon \ast\rightarrow \CondShape{X}$. Thus, for every pointed scheme, we can define \textit{condensed homotopy groups} by composition of the pointed version of (\ref{fun:condensed_shape}) with the condensed homotopy group functors on condensed anima (\cref{ex:condensedhomogrps}) giving functors
$$\pi_n^{\cond}\colon \Sch_{\mathrm{qcqs},
\ast}\rightarrow \Cond(\Ccal),\ (X,\xbar) \mapsto \pi_n^{\cond}(\CondShape{X},\xbar)\comma $$
where $\Ccal=\Grp$ for $n=1$ and $\Ccal=\Ab\Grp$ for $n\geq 2$. 
\end{recollection}
\begin{recollection}
Composition of (\ref{fun:condensed_shape}) with the functor $\pi_0\colon\CondAni\rightarrow \CondSet$ defines a \textit{condensed connected components} functor
$$\pi_0^{\cond}\colon \Sch_{\mathrm{qcqs}}\rightarrow \CondSet, \ X \mapsto \pi_0(\CondShape{X})\period$$
\end{recollection}
\begin{remark}\label[remark]{rem:conservative_fundamental_groups}
As stated in \cite[Example 2.4.7]{pyknoticI}, the functors $\pi_n^{\cond}$ and $\pi_0^{\cond}$ are collectively conservative: A morphism of condensed anima $X\rightarrow Y$ is an equivalence if and only if for every $n\geq 0$, the corresponding morphism on $\pi_n^{\cond}$ is an isomorphism in condensed sets, condensed groups or condensed abelian groups, respectively.
\end{remark}
Recall that for every qcqs scheme $X$ the set of connected components $\pi_0(X)$ (of its underlying Zariski space) carries a profinite topology via the quotient map $X\rightarrow \pi_0(X)$, see \stacks{0900}.
\begin{remark}
Let $X$ be a qcqs scheme. The \textit{condensed set of connected components} $\pi_0^{\cond}(X)$ only relies on the spectral space $|X|$ and does, in general, not coincide with the profinite set $\pi_0(X)$ of connected components. 
It is rather dependent on irreducible components than connected components of the scheme. Counterexamples arise from connected schemes with infinitely many irreducible components, see \cite[Section 4.2]{The_condensed_homotopy_type} and \cite[Example 2.3.33]{CatrinsThesis}.
\end{remark}

%

\newpage
We move on to the main part of this paper. It is structured as follows: After giving an alternative description of $\Gal(X)$ using work by Bhatt-Scholze and Lurie in \cref{sec:via_weakly_proobjects}, we provide a collection of (non-)examples of \textit{condensed contractible} schemes including the computation of the (underlying) condensed fundamental group of a general Dedekind domain in \cref{sec:condensed_contractible}, before we close with a classification of schemes having Galois categories with a terminal or an initial object in \cref{sec:Trivial_Shapes}.
\section{An alternative description of the Galois category}\label[section]{sec:via_weakly_proobjects}
The purpose of this section is to provide an alternative and more extensive description of the Galois category $\Gal(X)$ of a qcqs scheme $X$ that supports the idea of $\Gal(X)$ being a condensed and topos theoretic enrichment of the concept of (geometric) points of a scheme.
In particular, in \cref{prop:Galois_weaklycontractible}, we will show for every \textit{affine} scheme $X$ and $S\in \Extr$ the identification
    \begin{align}\label{eq:alternative_Gal}
    \Gal(X)(S) &\simeq \{\text{$S$-connected $w$-contractible rings over $X$ + $S$-connected maps}\} \period
  \end{align}
Here, $S$-connected means that we only consider $w$-contractible rings with $\pi_0(\Spec(A))=S$ and morphisms of rings $A\rightarrow B$ such that $\Spec(B) \rightarrow \Spec(A)$ induces the identity on $\pi_0$. This extends \cref{rem:points_strictlyhenselian} as a $w$-contractible ring is connected iff it is a strictly henselian local ring.
\subsection{Recollection on weakly contractible schemes}\label[section]{sec:weaklycontractible}
We need the following selection of definitions, mainly taken from \cite{MR3379634}.
\begin{definition}{\cite[Definition 2.1.1, Lemma 2.1.4]{MR3379634}}
Let $X$ be a spectral space. We call $X$ \textit{$w$-local} if it satisfies:
\begin{enumerate}
    \item The subspace $X_{\mathrm{cl}}\subset X$ of closed points is closed.
    \item Every connected component has a unique closed point.
\end{enumerate}
\end{definition}
\begin{remark}\label[remark]{rem:homeomorph_connectedcomp} Let $X$ be a $w$-local spectral space.
   The composition $X_{\mathrm{cl}}\rightarrow X\rightarrow \pi_0(X)$, where $\pi_0(X)$ carries its profinite topology, is a homeomorphism, see \cite[Lemma 2.1.4]{MR3379634}. 
\end{remark}
\begin{definition}\label[definition]{def:wlocal}
Let $X$ be a qcqs scheme. 
\begin{enumerate}
    \item We call $X$ \textit{$w$-local} if its underlying spectral space is $w$-local.
    \item We call $X$ \textit{$w$-strictly local} if it is $w$-local and every \'etale surjection $U\twoheadrightarrow X$ admits a section. 
\end{enumerate}
\end{definition}
\begin{remark}\label[remark]{rem:localring_connectedcomp}
The connected components of a $w$-local qcqs scheme $X$ correspond to local rings at closed points: By the assumption on every connected component to have a unique closed point, every other point in the connected component needs to specialize to this closed point. The local ring at the closed point is exactly the subscheme of its generalizations \stacks{01J7}.
\end{remark}
\begin{remark}
    As observed in \cite[Remark 2.45]{The_condensed_homotopy_type}, every $w$-strictly local qcqs scheme is affine as it is a retract of an affine scheme.
\end{remark}
\begin{remark}
A $w$-strictly local qcqs scheme equivalently is a $w$-local qcqs scheme such that all local rings at closed points are strictly henselian \cite[Lemma 2.2.9]{MR3379634}.
\end{remark}
\begin{definition}
Let $X$ be a qcqs scheme. An affine scheme $Y$ in the pro-\'etale site $\ProEt{X}$ is \textit{weakly contractible} if every pro-\'{e}tale covering morphism, i.e., every faithfully flat weakly \'etale map $V\rightarrow Y \in \ProEt{X}$, admits a section.   
\end{definition}
\begin{definition}
    A ring $A$ is \textit{$w$-contractible} if every faithfully flat ind-etale map $A\rightarrow B$, i.e.,  $B$ is a filtered colimit of \'{e}tale $A$-algebras, has a retraction. 
\end{definition}
\begin{recollection}\label[recollection]{rec:wstrictly_wcontractible}
Weakly contractible affine schemes identify with $w$-contractible rings. Moreover for $A$ $w$-contractible, $\Spec(A)$ is $w$-strictly local, \cite[Lemma 2.4.2]{MR3379634}, and a $w$-strictly local affine scheme $Y$ is weakly contractible if and only if $\pi_0(Y)\in \Extr$, see \cite[Lemma 2.5.8]{MR3379634}.
\end{recollection}
We give a simple example on how to create from a local ring $A$ and profinite set $S$ a $w$-strictly local affine scheme whose connected components are exactly given by $S$. If $S\in \Extr$, then this affine scheme is even weakly contractible.
\begin{example}\label[example]{ex:stricthensel_wstrictly}
    Let $X=\Spec(A)$ be the affine scheme of a local ring $A$.
    For every profinite set $S=\lim_{i \in I} S_i$ the scheme
    $$X_S\colonequals \lim_{i \in I}\coprod_{s\in S_i} X^{\mathrm{sh}}\comma $$ where $X^{\mathrm{sh}}=\Spec(A^{\mathrm{sh}})$ is the spectrum of the strict henselization of $A$, is a $w$-strictly local scheme above $X$ satisyfying $\pi_0(X_S)=S$. Indeed, every strictly henselian local ring defines a $w$-strictly local scheme with a unique connected component.
    As $w$-local spaces are closed under finite disjoint unions and limits \cite[Example 2.1.2., Lemma 2.1.9]{MR3379634}, the underlying spectral space of $X_S$ is $w$-local.
    Moreover, it is $\pi_0(X_S)=S$ by construction and every connected component corresponds to a copy of $X^{\mathrm{sh}}$.
    The closed subspace of closed points is then given by the scheme
    $$X_S^c= \lim_{i \in I} \coprod_{s\in S_i} \Spec(\kappa(m)^{\mathrm{sep}})\comma $$ where $\kappa(m)^{\mathrm{sep}}$ is a separable closure of the residue field at the maximal ideal $m\subset A$.
\end{example}
For a qcqs scheme $X$ and some $S\in \Extr$, every point $x\in X$ with local ring $\mathcal{O}_{X,x}$ gives rise to such a weakly contractible scheme. They are examples for \textit{$S$-connected weakly contractibles}. 
\begin{definition}\label[definition]{def:S_connected_weaklycontr}
Let $Y=\Spec(A)$ be a weakly contractible affine scheme.
\begin{enumerate}
    \item We call $Y$ an \textit{S-connected} weakly contractible object if $\pi_0(Y)=S\in \Extr$.
    \item  A morphism $Y\rightarrow Z$ of two $S$-connected weakly contractible affine schemes is called \textit{$S$-connected} if it induces the identity on $\pi_0$.
    \item For a given qcqs scheme $X$, we denote the category of \textit{$S$-connected weakly contractible affine schemes with $S$-connected maps over $X$}, i.e., with a structure map to $X$, by ${X_S^{\mathrm{wc}}}$.
\end{enumerate}
\end{definition}
\begin{example}
    The $\ast$-connected weakly contractible affine schemes identify with the affine schemes corresponding to strictly henselian local rings: a connected weakly contractible affine has a unique closed point with strictly henselian local ring which already forms the whole ring.
\end{example}
\begin{lemma}\label[lemma]{lem:wstrictly_coherentfunctor}
Let $X$ be a qcqs scheme and $Y\rightarrow X$ a weakly contractible affine scheme with $\pi_0(Y)=S$. The assignment $$(f\colon Y_{\mathrm{cl}} \subset Y \rightarrow X) \mapsto (f^*\colon \Sh(\Et{X}, \Set) \rightarrow \Sh(\Et{Y_{\mathrm{cl}}}, \Set))$$ defines a functor
$$G\colon ({X_{S}^{\mathrm{wc}}})^{\op} \rightarrow \Gal(X)(S)=\Fun^{\ast,\mathrm{coh}}(\Sh(\Et{X},\Set),\Sh(S,\Set))\period$$  
\end{lemma}
\begin{proof}  
    The closed subset of closed points $Y_{\mathrm{cl}}\subset Y$ equipped with the induced reduced subscheme structure inherits the property of being qcqs. (Indeed, it is the affine subscheme obtained by cutting out the Jacobsen radical of the ring underlying $Y$, see \cite[Lemma 2.2.3]{MR3379634}.)
    Morphisms between qcqs schemes induce coherent morphisms on topoi, compare \cite[Example 2.22]{Haine:1-localic}.
    Thus the by $Y_{\mathrm{cl}}\hookrightarrow Y \rightarrow X$ induced pullback morphism $f^*$ on sheaf topoi is coherent.
    We argue why $\Sh(\Et{Y_{\mathrm{cl}}}, \Set)\simeq \Sh(S,\Set)$.
    By construction, $Y_{\mathrm{cl}}$ is an everywhere strictly local scheme meaning that the local ring at every geometric point is strictly henselian.\footnote{The notion of an \textit{everywhere strictly local} scheme will come up again later in \cref{sec:Trivial_Shapes}, see \cref{def:everywhere_strictly:local}.}
    For such a scheme, the \'{e}tale and the Zariski topoi agree by \cite[Corollary 2.5]{MR3649361}. Moreover, we have a homeomorphism $Y_{\mathrm{cl}}\simeq \pi_0(Y)=S$ of topological spaces by \cref{rem:homeomorph_connectedcomp}.
    This implies that the Zariski topos on $Y_{\mathrm{cl}}$ is equivalent to the sheaf topos on $S$.
    A morphism $V\rightarrow Y \in {X_{S}^{\mathrm{wc}}}$ gives a morphism $g\colon S=\pi_0(V) \rightarrow \pi_0(Y)=S$ that is by the definition of morphisms in ${X_{S}^{\mathrm{wc}}}$ the identity on $S$. This morphism induces a coherent algebraic morphism $$g^*\colon \Sh(\Et{Y_{\mathrm{cl}}}, \Set)\simeq \Sh(S,\Set)\rightarrow \Sh(S,\Set)\simeq \Sh(\Et{V_{\mathrm{cl}}}, \Set)\period $$  
    Post-composition with $g^*$ describes the natural transformation between the functors $$f_1^*,
    f_2^*\colon \Sh(\Et{X}, \Set)\rightarrow \Sh(S,\Set)$$ which are induced by the morphisms $f_1\colon Y_{\mathrm{cl}} \hookrightarrow Y \rightarrow X$ and $f_2 \colon V_{\mathrm{cl}} \hookrightarrow V \rightarrow X$.
\end{proof}
\subsection{Relation to Lurie's ultracategories}
We recall work on ultracategories by Lurie \cite{Ultracategories} in order to connect it to $\Gal(X)$. More precisely, it provides a description for the functor category $\Fun^*(\Sh(\Et{X},\Set),\Sh(S,\Set))$ of \textit{all} pullback functors from the \'etale topos of a qcqs scheme $X$ to the sheaf topos of a profinite set $S$. We will use this in the next subsection to show the claimed statement (\ref{eq:alternative_Gal}) on the full subcategory  of coherent functors $$\Gal(X)(S)\hookrightarrow\Fun^*(\Sh(\Et{X}, \Set),\Sh(S,\Set))\period $$
\begin{recollection}
Recall the notion of a pretopos as an exact and extensive category, see, e.g., \cite[Definition A.4.1]{Ultracategories}.
It presents a more general variant of a topos where we only ask for the existence of finite limits and colimits.
For $\Ccal$ and $\Dcal$ pretopoi, a functor $F \colon \Ccal \rightarrow \Dcal$ is a pretopos functor if it preserves finite limits, finite coproducts, and effective epimorphism.
We let $\Fun_{\mathrm{pre}}(\Ccal,\Dcal)\subset \Fun(\Ccal, \Dcal)$ denote the full subcategory spanned by the pretopoi functors.
\end{recollection}
\begin{recollection}
Let $X$ be a qcqs scheme.
The subcategory of coherent objects in the coherent topos $\Sh(\Et{X}, \Set)$ is a small pretopos \cite[Corollary C.5.14]{Ultracategories}. We denote by $X_{\et}^{\coh}$ this pretopos endowed with the coherent topology and have $\Sh(X_{\et}^{\coh})=\Sh(\Et{X}, \Set)$, see \cite[Proposition C.6.4, Theorem C.6.5]{Ultracategories}.
Moreover, for every topos $\mathcal{D}$ there is an equivalence 
\begin{align}\label{eq:pre_topoi_morphisms}
\Fun_{\mathrm{pre}}(X_{\et}^{\coh}, \mathcal{D})\simeq \Fun^{*}(\Sh(\Et{X},\Set),\mathcal{D})\period
 \end{align}
of pretopoi morphisms and algebraic morphisms of topoi, see \cite[Corollary C.3.6]{Ultracategories}.
\end{recollection}
\begin{definition}
Let $X$ be a qcqs scheme. We define the \textit{condensed category of (non-coherent) points} of (the \'etale topos of) $X$ by the assignment 
$$\Pt(X)\colon \Extr \rightarrow \Cat, \  S \mapsto \Fun^*(\Sh(\Et{X},\Set),\Sh(S,\Set)).$$   
\end{definition}
\begin{remark}
As the name suggests, the condensed category of (non-coherent) points $\Pt(X)$ differs from $\Gal(X)$ only in the regard that we also consider non-coherent functors. Haine also compares these two condensed categories and proves that the inclusion $$\Gal(X) \hookrightarrow \Pt(X)$$ induces an equivalence on condensed classifying anima, see \cite[Theorem 2.20]{Classifyinganima}.
\end{remark}
\begin{remark}{\textbf{Condensed category of points as an ultracategory}}\\
For a qcqs scheme $X$, the condensed category of (non-coherent) points $\Pt(X)$ finds a description as an ultracategory in the category $\mathrm{Ult^L}$ of \textit{ultracategories with left ultrafunctors}, see \cite[Definition 1.4.1]{Ultracategories}. There are fully faithful embeddings 
$$\{\text{coherent $1$-topoi, geometric morphisms}\} \hookrightarrow \mathrm{Ult^L} \hookrightarrow \{\text{Stacks of categories on \Comp}\}\comma$$
compare \cite[Remark 2.2.9., Remark 4.3.4]{Ultracategories}.
Here, the first embedding sends a coherent topos $\mathcal{X}$ to the ultracategory $$\mathrm{Mod}(\mathcal{X}^{\coh})\colonequals \Fun_{\mathrm{pre}}(\mathcal{X}^{\coh}, \Set)$$ of \textit{models} on its pretopos of coherent objects $\mathcal{X}^{\coh}$ which coincides, on the level of categories, with the (1-)category of points $\Pt(\mathcal{X})$.
The second functor sends an ultracategory $\mathcal{M}$ to the stack on compact Hausdorff spaces $\Comp$ (with respect to the Grothendieck topology given by finite jointly surjective families) given by the assignment
$$X\mapsto\Fun^{\mathrm{LUlt}}(X,\mathcal{M}) \comma$$ 
see \cite[Proposition 4.1.5]{Ultracategories}.
Note that the category of stacks of categories on the site $\Comp$ identifies with the category $\Cond(\Cat)$ of condensed categories. From this perspective, the \textit{condensed category of (non-coherent) points} $\Pt(X)$ is the image of the ultracategory $\mathrm{Mod}(X_{\et}^{\coh})$ under the composition of embeddings above. More generally, Lurie's work shows that there is a fully faithful functor of \textit{condensed points} on the category of all coherent $1$-topoi
\begin{equation}\label{eq:fullyfaithfulpoints}
\Pt \colon \{\text{coherent $1$-topoi, geometric morphisms}\} \hookrightarrow \Cond(\Cat)\period   
\end{equation}
\end{remark}
We will first describe the condensed category of (non-coherent) points, before passing over to $\Gal(X)$.
We do not explicitly work with the notion of ultracategories here but use results and notation provided by Lurie in \cite[Section 6.2]{Ultracategories}.
\begin{recollection}\label[recollection]{rec:def_weaklyproj}
Let $\Ccal$ be a pretopos and $\Pro(\Ccal)$ its category of \textit{pro-objects} whose objects are left exact functors $ \fromto{\Ccal}{\Ani} $.
An object $W\in \Pro(\Ccal)$ is called \textit{weakly projective}, see \cite[Definition 6.2.2]{Ultracategories}, if for every effective epimorphism $f\colon C\rightarrow D$ in the pretopos $\Ccal$, composition with $f$ induces a surjection 
\begin{align}\label{weakly_proj}
\Hom_{\Pro(\Ccal)}(W,C)\rightarrow \Hom_{\Pro(\Ccal)}(W,D)\period
\end{align}    
 They can be identified with finite limits and effective epimorphisms preserving functors $\Ccal \rightarrow \Set$.
\end{recollection}
\begin{notation}
Let $\Ccal$ be a pretopos. We denote by \textbf{$\Pro^{\mathrm{wp}}(\Ccal)$} the subcategory spanned by the \textit{weakly projective objects} in $\Pro(\Ccal)$ .   
\end{notation}
The following definition is given by Lurie in \cite[Definition 6.3.8]{Ultracategories}.
\begin{definition}
Let $\Ccal$ be a pretopos.
We denote by \textbf{$\mathrm{Stone}_{\Ccal}$} the category with 
\begin{itemize}
    \item[1.] \textbf{Objects:} An object is a pair $(X,\mathcal{O}_X)$, with $X\in \ProFin$ and $\mathcal{O}_X\colon \Ccal \rightarrow \Sh(X, \Set)$ a pretopos morphism.
    \item[2.] \textbf{Morphisms:} A morphism $(X,\mathcal{O}_X)\rightarrow(Y,\mathcal{O}_Y)$ is a pair $(f,\alpha)$ with $f$ a continuous map of profinite sets and $\alpha\colon f^*\mathcal{O}_Y\rightarrow \mathcal{O}_X$ a natural transformation of functors $\Ccal \to \Sh(X,\Set)$ induced by $f$.  
\end{itemize}
\end{definition}
\begin{remark}
Let $\Ccal$ be a pretopos and $X\in \Top$ a topological space. In \cite{Ultracategories}, pretopos morphisms $\mathcal{O}_X\colon \Ccal \rightarrow \Sh(X, \Set)$ are called \textit{$X$-models}. They generalize the notion of \textit{models} given by pretopoi morphisms $\Ccal \rightarrow \Set$. For each point $x\in X$, composition of an $X$-model with the pullback along $\{x\}\hookrightarrow X$ identifies with a model of $\Ccal$ or, in other words, with a point of the topos $\Sh(\Ccal)$. One can think of an $X$-model of a pretopos $\Ccal$ as a continuous family of models of $\Ccal$, parametrized by the space $X$, compare \cite[Example 6.3.5, Remark 6.3.6, Remark 6.3.7]{Ultracategories}.
\end{remark}
Our most important ingredient is the following theorem by Lurie showing that every weakly projective pro-object on a pretopos $\Ccal$ can be viewed as a continuous family of points of the topos $\Sh(\Ccal)$ (where $\Ccal$ is equipped with the coherent topology) parametrized by a profinite set $X$.
\begin{theorem}{\cite[Theorem 6.3.14]{Ultracategories}}
Let $\Ccal$ be a pretopos.
There is an equivalence of categories
\begin{align}\label{eq:Ultracat_Stone}
\mathrm{Stone}_{\Ccal} \simeq \Pro^{\mathrm{wp}}(\Ccal) 
\end{align}
induced by post-composition with the global sections $\Gamma_X\colon \Sh(X,\Set) \rightarrow \Set$ for $X\in \ProFin$.    
\end{theorem}
\begin{recollection}
There are functors
\begin{align}
&F_1\colon\mathrm{Stone}_{\Ccal} \rightarrow \ProFin \text{ and } F_2\colon \Pro(\Ccal)\rightarrow \ProFin \end{align} which are compatible with the equivalence (\ref{eq:Ultracat_Stone}), see  \cite[Construction 6.4.3]{Ultracategories}.
While the first functor just forgets about the pretopos morphism, the second one is the left adjoint of the unique cofiltered limit preserving extension of $$\FinSet \rightarrow \Ccal, \ S \mapsto \coprod_{s \in S} \ast_{\Ccal}\period$$
\end{recollection}
\begin{notation}
For a fixed $S\in \Extr$, we consider the subcategory $(S,\mathrm{id})\subset \ProFin$ only consisting of $S$ and the identity map $S\rightarrow S$.  
\end{notation}
We will take the fiber of this subcategory under $F_1\colon \mathrm{Stone}_{\Ccal} \rightarrow \ProFin.$  
With the right choice on $\Ccal$, this provides us with the connection to $\Pt(X)(S)=\Fun^*(\Sh(\Et{X},\Set),\Sh(S,\Set))$.
\begin{observation}
Let $S\in \Extr$. Setting $\mathcal{D}=\Sh(S, \Set)$ in the equivalence (\ref{eq:pre_topoi_morphisms}), we can identify
$$\Fun_{\mathrm{pre}}(X_{\et}^{\coh}, \Sh(S, \Set))\simeq \Pt(X)(S)\period $$
In other words $\Pt(X)(S)$ is the category of $S$-models of $X_{\et}^{\coh}$.
Now consider the preimage of the subcategory $(S,\mathrm{id})\subset\ProFin$ under the forgetful functor $$F_1\colon \mathrm{Stone}_{X_{\et}^{\coh}}\rightarrow \ProFin \period $$  This is the subcategory of $\mathrm{Stone}_{X_{\et}^{\coh}}$ given by  pairs $(S,\mathcal{O}_S)$ where $S$ is fixed. Note that a morphism $\mathcal{O}_S \rightarrow \mathcal{O}_S'$ in this subcategory corresponds to a natural transformation $\alpha\colon \mathcal{O}_S'\rightarrow \mathcal{O}_S$ of pretopoi morphisms. Thus this subcategory identifies with the opposite category $\Pt(X)(S)^{\op}.$ 
\end{observation}
By the last observation, under the equivalence $(\ref{eq:Ultracat_Stone})$ the opposite category $$\Fun_{\mathrm{pre}}(X_{\et}^{\coh}, \Sh(S, \Set))^{\op}=\Pt(X)(S)^{\op}$$ corresponds to a subcategory of $\Pro^{\mathrm{wp}}(X_{\et}^{\coh})$ that can be described as follows:
\begin{definition}\label[definition]{def:Kconnectweakpro}
Let $X$ be a qcqs scheme.
For a fixed $S\in \Extr$, let $\Pro^{\mathrm{wp}}(X_{\et}^{\coh})_S$ denote the subcategory of $\Pro^{\mathrm{wp}}(X_{\et}^{\coh})$ consisting of 
\begin{itemize}
    \item[1.] \textbf{Objects:} Weakly projective objects $W\in \Pro(X_{\et}^{\coh})$ such that $F_2(W)=S$.
    \item[2.] \textbf{Morphisms:} Morphisms in $\Pro^{\mathrm{wp}}(X_{\et}^{\coh})$ inducing the identity under $F_2$.
\end{itemize}
We refer to this category as the category of \textit{$S$-connected weakly projective pro-objects} of $X_{\et}^{\coh}$.
\end{definition}
\begin{observation}
The category $\Pro^{\mathrm{wp}}(X_{\et}^{\coh})_S$ is the preimage of $(S, \mathrm{id})\subset \ProFin$ under the functor
$F_2\colon \Pro(X_{\et}^{\coh}) \rightarrow \ProFin.$    
\end{observation}
Summarizing, we find the following alternative description of the condensed category of (non-coherent) points $\Pt(X).$
\begin{corollary}\label[corollary]{cor:alter_Gal}
Let $X$ be a qcqs scheme.
For every $S\in \Extr$, there is an equivalence of categories
$$\Pt(X)(S)\simeq \Pro^{\mathrm{wp}}(X_{\et}^{\coh})_S^{\op}$$
induced by postcomposition with the global sections functor $\Gamma_S\colon \Sh(S,\Set) \rightarrow \Set$ (and further restriction to coherent objects).
\end{corollary}
\begin{remark}
Observe that all conclusions from this subsection are valid in the general setting of a coherent topos: For a coherent topos $\mathcal{X}$, the subcategory of coherent objects $\mathcal{X}^{\coh}$ is a pretopos \cite[Corollary C.5.14]{Ultracategories} and $\mathcal{X}=\Sh(\mathcal{X}^{\coh})$ where $\mathcal{X}^{\coh}$ carries the coherent topology \cite[Proposition C.6.4, Theorem C.6.5]{Ultracategories}. Let $\Pt(\mathcal{X})$ be the image of $\mathcal{X}$ under the fully faithful functor (\ref{eq:fullyfaithfulpoints}). Then the counterpart of \cref{cor:alter_Gal} gives an equivalence of categories
$$\Pt(\mathcal{X})(S)\simeq \Pro^{\mathrm{wp}}(\mathcal{X}^{\coh})_S^{\mathrm{op}}\period $$
This statement is a condensed enhancement of the classical topos theoretic result, e.g., stated in \cite[Proposition 1.4]{GABBER20154667}, for categories $\Ccal$ with finite limits and a topology $\tau$ claiming that
\begin{align}
    \{\text{fibre functors on } \Sh_{\tau}(\Ccal)\} \simeq \{\text{$\tau$-local pro-objects in $\Ccal$}\}
\end{align}
Note that the definition of $\tau$-local pro-objects in $\Ccal$ \cite[Definition 1.2]{GABBER20154667} exactly recovers those weakly projective pro-objects in $\Ccal$ that also preserve finite coproducts. This is the case if and only if the pro-object is connected meaning that it corresponds to $S=\ast$, as observed in \cite[Proposition 6.3.12 (3)]{Ultracategories}.
\end{remark}
In the next subsection, we will show that the equivalence in \cref{cor:alter_Gal} restricts (for an affine scheme $X$) on $\Gal(X)(S)$ to \textit{projective} objects in $\Pro(\mathcal{X}^{\coh})$ which in turn identify with $S$-connected $w$-contractible rings over $X$. This will give the claimed result (\ref{eq:alternative_Gal}).
\subsection{Galois category in terms of weakly contractibles}
In this subsection, we use the content of the previous two subsections to give an alternative description of the Galois category of a scheme in terms of weakly contractible affine schemes or, equivalently, $w$-contractible rings. For simplicity, we restrict ourselves to affine base schemes $X$.
We are aiming to prove the following result which was already stated in (\ref{eq:alternative_Gal}):
\begin{theorem}\label[theorem]{prop:Galois_weaklycontractible}
    Let $X$ be an affine scheme. For every $S\in \Extr$, the equivalence of categories from \cref{cor:alter_Gal} restricts to an equivalence
    \begin{align*}
    \Gal(X)(S) &\simeq \{\text{$S$-connected $w$-contractible rings over $X$ + $S$-connected  maps}\} \period    
  \end{align*}
\end{theorem}
Recall the definition of the category $X^{\mathrm{wc}}_S$ of $S$-connected weakly contractible affine schemes over $X$ from \cref{def:S_connected_weaklycontr}. We first show that this category identifies with a full subcategory of $\Pro^{\mathrm{wp}}(X_{\et}^{\coh})_S$ namely the one of \textit{projective} objects. Then it remains to argue why these objects correspond exactly to those functors in $\Fun_{\mathrm{pre}}(X_{\et}^{\coh}, \Sh(S, \Set))$ which are coherent, i.e., whose image lands in the subcategory $\Sh(S,\Set)^{\mathrm{coh}}\subset \Sh(S,\Set)$ of coherent objects.
\begin{remark}
Let $X$ be an affine scheme. We denote by $\mathrm{Aff}\Et{X}$ all schemes \'etale affine over $X$.
\begin{enumerate}
    \item Every affine \'etale scheme over $X$ gives a coherent (representable) sheaf in the \'etale topos and hence we have a full subcategory relation $$\Pro(\mathrm{Aff}\Et{X})\subset \Pro(X_{\et}^{\coh})\period$$
    \item As every $C\in X_{\et}^{\coh}$ admits an effective epimorphism $C_0\twoheadrightarrow C$ with $C_0\in \mathrm{Aff}\Et{X}$, all weakly projective pro-objects over $X_{\et}^{\coh}$ already lie in $\Pro^{\mathrm{wp}}(\mathrm{Aff}\Et{X})$, see \cite[Remark 6.2.9]{Ultracategories}.
    \item  As we assume $X$ to be affine, the category $\Pro(\mathrm{Aff}\Et{X})$ of \textit{pro-\'etale affine schemes over $X$} is simply the category of affine schemes pro-\'etale over $X$, see \cite[Remark 4.2.5]{MR3379634}.
    \item Restriction of the functor        $F_2\colon\Pro(X_{\et}^{\coh}) \rightarrow \ProFin$ to $\Pro(\mathrm{Aff}\Et{X})$ coincides with $\pi_0$ on the level of schemes by uniqueness of left adjoints.
\end{enumerate}
\end{remark}
By the previous remark, we can replace $\Pro^{\mathrm{wp}}(X_{\et}^{\coh})$ in the discussion by $\Pro^{\mathrm{wp}}(\mathrm{Aff}\Et{X})$ and think about its objects as proper affine schemes in the pro-\'etale site over $X$. 
\begin{recollection}
    Recall that an object $P$ in a regular category, such as a pretopos $\Ccal$ or it pro-category $\Pro(\Ccal)$, see \cite[Corollary 6.1.20]{Ultracategories}, is \textit{projective} if the hom-functor $\Hom(P,-)$ preserves effective epimorphisms or, equivalently, if any effective epimorphism $Q\twoheadrightarrow P$ admits a section \cite[Proposition 6.2.1]{Ultracategories}. Every projective object in $\Pro(\Ccal)$ is weakly projective over $\Ccal$.
\end{recollection}
\begin{notation}
    Let $X$ be an affine scheme. As we have $\Pro^{\mathrm{wp}}(X_{\et}^{\coh})=\Pro^{\mathrm{wp}}(\mathrm{Aff}\Et{X})$, every projective pro-object over the pretopos $X_{\mathrm{et}}^{\coh}$ is a projective object in $\Pro(\mathrm{Aff}\Et{X})$. We denote the subcategory of projective pro-objects by $\Pro^{\mathrm{proj}}(\mathrm{Aff}\Et{X})$.
    For every $S\in \Extr$, we again denote by $$\Pro^{\mathrm{proj}}(\mathrm{Aff}\Et{X})_S\subset \Pro^{\mathrm{proj}}(\mathrm{Aff}\Et{X})$$ the (non-full) subcategory of affine schemes pro-\'etale over $X$ with connected components equal to $S$ and morphisms those inducing the identity on $\pi_0$.
\end{notation}
\begin{lemma}\label[Lemma]{lem:weakly_contractible_projective}
    Let $X$ be an affine scheme. 
    We have an equivalence of categories
    $$\Pro^{\mathrm{proj}}(\mathrm{Aff}\Et{X})\simeq \{\text{weakly contractible affine schemes over $X$}\}\period $$
    For every $S\in \Extr$, this induces
    $$\Pro^{\mathrm{proj}}(\mathrm{Aff}\Et{X})_S\simeq X_S^{\mathrm{wc}}\period $$
\end{lemma}
\begin{proof}
    The first part is immediate from the definition of weakly contractible affine schemes by the property that every pro-\'etale surjection onto them admits a section. 
    The second part then follows by comparing the extra assumptions on objects and maps that are indicated by the subscript $S$: Both sides exactly restrict to those objects whose connected components identify with $S$ and those morphisms such that composition with $\pi_0$ is the identity on this space.
\end{proof}
\begin{proposition}\label[proposition]{prop:Gal_wcontrrings}
 Let $X$ be an affine scheme. For every $S\in \Extr$, composition with the global sections functor $\Gamma_S\colon \Sh(S,\Set)\rightarrow \Set$ induces an equivalence of categories
  \begin{align*}
  \Gal(X)(S) &\simeq \Pro^{\mathrm{proj}}(\mathrm{Aff}\Et{X})_S^{\op}\simeq ({X_{S}^{\mathrm{wc}}})^{\op}\period    
  \end{align*}
  The reverse direction is given by the functor described in \cref{lem:wstrictly_coherentfunctor}.    
\end{proposition}
\begin{proof}
Let $f^*\in \Gal(X)(S)$. Then $f^*$ is a coherent functor $\Sh(\Et{X},\Set)\rightarrow \Sh(S,\Set)$, which means that it restricts to a morphism $X_{\et}^{\coh}\rightarrow \Sh(S,\Set)^{\mathrm{coh}}$ on pretopoi of coherent objects. Recall that $\Sh(S, \Set)^{\coh}$ identifies with the category of locally constant constructible sheaves by $\Sh(S,\Set)$ being a profinite topos \cite[Example E.2.4.6., Proposition E.2.5.6]{SAG}.
The equivalence in \cref{cor:alter_Gal} we are restricting from is given by post-composition with the global sections functor $\Gamma_S\colon \Sh(S,\Set)\rightarrow \Set$. We are aiming to show that the (scheme representing the) resulting functor 
$X_{\et}^{\mathrm{coh}}\rightarrow \Set$ constitutes a projective object in $\Pro(\mathrm{Aff}\Et{X})$ or, equivalently, in $\Pro(X_{\et}^{\mathrm{coh}})$.
Reformulating the definition of projectivity, this means nothing else than that its extension to $\Pro(X_{\et}^{\mathrm{coh}})$ preserves effective epimorphisms.
This extension coincides with the composition
\begin{align}\label{eq:composition_functor}
 \Pro(X_{\et}^{\mathrm{coh}}) \xrightarrow{\Pro(f^*)}\Pro(\Sh(S,\Set)^{\mathrm{coh}})\xrightarrow{\Pro(\Gamma_S)} \Set \period   
\end{align}
Recall that a pretopos morphism preserves effective epimorphisms and that effective epimorphisms in pro-categories are effective epimorphisms level-wise \cite[Proposition 6.1.15]{Ultracategories}.
Hence, the property of preserving effective epimorphisms extends from $X_{\et}^{\mathrm{coh}}\rightarrow \Sh(S,\Set)^{\mathrm{coh}}$ to the unique cofiltered limit preserving extension $\Pro(f^*)$ in (\ref{eq:composition_functor}).
The (extended) global sections functor $\Pro(\Gamma_S)$ also preserves effective epimorphisms.
This can, for example, be argued by the fact that $\Pro(\Sh(S,\Set)^{\mathrm{coh}})$ is a full subcategory of the slice category $\CondSet_{/S}$, e.g., see \cite[Remark 3.27, Proposition 3.32]{MR4574234}.
As $S$ is an extremally disconnected profinite set, it is known to be a projective object in $\CondSet$, compare \cref{rem:extdis_projective}.
This property is retained after restriction to the subcategory $\Pro(\Sh(S,\Set)^{\mathrm{coh}})$ and thus the evaluation on $S$ preserves effective epimorphisms.
We have argued that restriction of the equivalence given in \cref{cor:alter_Gal} $$\Pt(X)(S)\simeq \Pro^{\mathrm{wp}}(X_{\et}^{\coh})_S^{\mathrm{op}}\simeq \Pro^{\mathrm{wp}}(\mathrm{Aff}\Et{X})_S^{\mathrm{op}}$$ to $\Gal(X)(S)\subset \Pt(X)(S)$ lands in the subcategory of projective objects $\Pro^{\mathrm{proj}}(\mathrm{Aff}\Et{X})_S^{\mathrm{op}}$.
It remains to see whether every projective object is the image of a coherent functor.
For that, we show that for a scheme in $\Pro^{\mathrm{proj}}(\mathrm{Aff}\Et{X})_S\simeq {X_{S}^{\mathrm{wc}}}$ the coherent functor defined as in \cref{lem:wstrictly_coherentfunctor} provides a preimage.
In other words, we want to prove that for a weakly contractible affine scheme $Y\rightarrow X$ with $\pi_0(Y)=S$ composition of the by $Y_{\mathrm{cl}}\subset Y \rightarrow X$ induced coherent geometric morphism 
$$f^*\colon \Sh(\Et{X}, \Set) \rightarrow \Sh(\Et{Y_{\mathrm{cl}}}, \Set)\simeq \Sh(S,\Set)$$ with the global sections functor $\Gamma_S$ is represented by the scheme $Y$.
To make this precise, by viewing $Y\in \Sh(\Pro(\mathrm{Aff\Et{X}}),\Set)=\Sh(\ProEt{X}, \Set) $ and $\Sh(\Et{X}, \Set)\subset \Sh(\ProEt{X}, \Set)$, we want $$\Gamma \circ f^*=\Hom_{\Sh(\ProEt{X}, \Set)}(Y,-).$$
By definition, $f^*$ factors over $\Sh(\Et{Y}, \Set)$. We extend $f^*$ to a functor on pro-\'etale topoi
$$\Sh(\ProEt{X}, \Set) \xrightarrow{g^*} \Sh(\ProEt{Y}, \Set) \xrightarrow{h^*} \Sh(\ProEt{Y_{\mathrm{cl}}}, \Set)$$ by first restricting to coherent objects, then extending to pro-objects and finally taking sheaves with respect to the coherent topology, compare \cite[Example 7.1.7]{Ultracategories}.
As $Y\rightarrow X$ is pro-\'etale, the functor $g^*$ is just restriction of a sheaf to the subsite $\ProEt{Y}\subset \ProEt{X}$. Recall that $Y_{\mathrm{cl}}\hookrightarrow Y$ is a closed immersion. For closed immersions of affine schemes, the induced geometric morphism on pro-\'etale topoi is well-studied, compare \stacks{09A9}.
In particular, we can deduce from \stacks{09AC} that for every $F\in \Sh(\ProEt{Y}, \Set)$ it is $h^*(F)(Y_{\mathrm{cl}})=F(Y)$ as by assumption $Y$ is affine, weakly contractible and every point of it specializes to a point in $Y_{\mathrm{cl}}$.
Combining all of this, the image of a sheaf $F \in \Sh(\Et{X}, \Set)\subset \Sh(\ProEt{X}, \Set)$ under $\Gamma \circ f^*$ is given by $$(\Gamma \circ f^*)(F)=(\Gamma \circ h^* \circ g^*)(F)=(h^*\circ g^*)(F)(Y_{\mathrm{cl}})=g^*(F)(Y)=F(Y)\period$$ Using the Yoneda identification $F(Y)=\Hom_{\Sh(\ProEt{X}, \Set)}(Y,F)$, this finishes the proof.
\end{proof}
We are basically done with proving our aimed result:
\begin{proof}[Proof of \cref{prop:Galois_weaklycontractible}]
By passing in the last proposition from ${X_{S}^{\mathrm{wc}}}$ to its opposite category of rings, one obtains 
$$\Gal(X)(S)\simeq \{\text{$S$-connected $w$-contractible rings over $X$ + $S$-connected maps}\} $$
giving the statement of \cref{prop:Galois_weaklycontractible}.   
\end{proof}
We finish this subsection with showing that for a local ring $A$ and $S=\lim_{i \in I} S_i\in \Extr$, the ring 
$$A_S\colonequals \colim_{i \in I^{\op}} \prod_{s\in S_i} A^{\mathrm{sh}}$$
underlying the scheme $\Spec(A)_S$ defined in \cref{ex:stricthensel_wstrictly} takes over a special role in $\Gal(X)(S)$.
\begin{proposition}\label[proposition]{prop:alter_Gal_prop}
Let $A$ be a local ring and $X=\Spec(A)$. The ring $$A_S\colonequals \colim_{i \in I^{\op}} \prod_{s\in S_i} A^{\mathrm{sh}}$$ is a weakly initial object in $\Gal(X)(S)$, i.e., it admits a map into every other object.
Moreover, it is initial if and only of $A$ is strictly henselian.
\end{proposition}
\begin{proof}
First note that by \cref{ex:stricthensel_wstrictly} and \cref{prop:Galois_weaklycontractible}, the ring $A_S$ provides an object in $\Gal(X)(S)$. Let $B\in \Gal(X)(S)$ an $S$-connected $w$-contractible ring over $X=\Spec(A)$. The association $S\mapsto X_S\colonequals\lim_{i \in I} \coprod_{s \in S_i} X $ is a limit preserving functor from profinite sets to affine schemes pro-\'etale over $X$, see \cite[Example 4.1.9]{MR3379634}. It admits a left adjoint given by the connected component functor, e.g., see \cite[Lemma 2.2.8]{MR3379634} or \cite[Lemma A.5.26]{CatrinsThesis}. Thus the identity map $\pi_0(\Spec(B))=S\rightarrow S$ corresponds to a map $\Spec(B) \rightarrow X_S$ which in turn corresponds to a ring map $\colim_{i \in I^{\op}} \prod_{s \in S_i} A\rightarrow B$. Now the local ring map $A^{\mathrm{sh}}\rightarrow A$ coming from strict henselization lifts to a map $A_S\rightarrow \colim_{i \in I^{\op}} \prod_{s \in S_i} A$. Composition then defines a map $A_S \rightarrow B$ in $\Gal(X)(S)$.
For the stricter version, we assume $A$ to be a strictly henselian local ring. Then we have $X_S=\Spec(A_S)$ and we directly obtain the map $A_S \rightarrow B$ by the adjunction described above. Because of the adjunction, it is the unique map inducing the identity on $\pi_0$ and thus lying in $\Gal(X)(S)$. If $A$ is not strictly henselian, every automorphism $a$ of the strict henselization $A^{\mathrm{sh}}$ fixing the ring $A$ in the sense that the composition $A^{\mathrm{sh}}\xrightarrow{a} A^{\mathrm{sh}} \rightarrow A$ coincides with the given map $A^{\mathrm{sh}}\rightarrow A$ yields another map from $A_S$ to $B$.
\end{proof}
Roughly speaking, the fact that for a general local ring we only obtain a \textit{weakly} initial object is due to the existence of non-trivial automorphisms of the closed point corresponding to objects in the absolute Galois group of its (non-separably closed) residue field.
\begin{remark}
    One can show that, in the situation of \cref{prop:alter_Gal_prop}, the ring $A_S$ even provides a (weakly) initial object in the category $\Pt(X)(S)$. This is in line with the by Haine in \cite[Theorem 2.20]{Classifyinganima} proven fact that for all $S\in \Extr$ the inclusion $\Gal(X)(S)\hookrightarrow \Pt(X)(S)$ admits a left adjoint inducing an equivalence of classifying condensed anima $$\CondShape{X}=B^{\cond}\Gal(X)\simeq B^{\cond}\Pt(X)\period $$
\end{remark}
\section{On condensed contractible schemes}\label[section]{sec:condensed_contractible}
We first recap how the notions of connected and contractible categories naturally translate to the condensed setting. Then we collect examples of schemes which are known to be \textit{condensed contractible}, i.e., the condensed homotopy type is contractible.
\begin{recollection}
An \category $\Ccal$ is \textit{connected}, respectively, \textit{contractible,} if the classifying anima $\Bup\Ccal \in \Ani$ is a connected, respectively, (weakly) contractible Kan complex. 
\end{recollection}
\begin{definition}
We say that a condensed \category $\Ccal$ is \textit{connected}, if for all $K\in \Extr$ the \categories $\Ccal(K)\in \Catinfty$ are connected.
Equivalently, if the classifying condensed anima $\Bup\Ccal\in \CondAni$ is \textit{connected}, i.e., its condensed connected components are trivial $\pi_0(\Bup\Ccal)=\ast$.
Similarly, a condensed category $\Ccal \in \Catinfty$ is \textit{contractible} if the evaluation $\Ccal(K)$ at all $S \in \Extr$ is contractible, i.e., the anima $\Bup\Ccal(K)=\ast$ are contractible.
\end{definition}
\begin{definition}\label[definition]{def:condensed_contractible}
Let $X$ be a qcqs scheme. 
\begin{enumerate}
    \item We call the scheme $X$ \textit{condensed connected} if the condensed connected components are trivial, i.e., we have $\pizerocond(X)=\ast$.
    \item We call the scheme $X$ \textit{condensed contractible} if the condensed homotopy type is contractible, i.e., it is $\CondShape{X}=\ast.$
\end{enumerate}
\end{definition}
\begin{remark}
Let $X$ be a qcqs scheme. By definition, $X$ is condensed connected, respectively, condensed contractible if and only if the condensed category $\Gal(X)\in \Cond(\Cat)$ is connected, respectively, contractible. Whenever $\CondShape{X}$, respectively $\Gal(X)$, is contractible it is connected. The reverse direction is not true.
\end{remark}
\begin{remark}\label[remark]{rem:trivial_on_fundgroups}
Under notice of \cref{rem:conservative_fundamental_groups}, a scheme is condensed contractible if and only if it is condensed connected an all its condensed fundamental groups (which then are independent of the base point) are trivial.  \end{remark}
\begin{proposition}\label{cor:trivialconshape}
Let $X$ be a qcqs scheme. The following assertions are equivalent:
\begin{itemize} 
\item[(i)] The scheme $X$ is condensed contractible.
\item[(ii)] The pullback functor $$f^*\colon\CondAni\simeq \Functs(\pt,\ICond(\Ani)) \rightarrow \Functs(\Gal(X),\ICond(\Ani))\simeq\Xproethyp$$ along $\Gal(X) \rightarrow \pt$ is fully faithful.
\end{itemize}
\end{proposition}
\begin{proof}
By \cite[Corollary~1.2]{MR4574234}, there is a natural equivalence
\begin{equation*}
	\Functs(\Gal(X),\ICond(\Ani)) \simeq \Xproethyp \period
\end{equation*}
and precomposition with the localization map $ b \colon \Gal(X) \to \BcondGal(X) $ defines a fully faithful functor
	\begin{equation*}
		\begin{tikzcd}
			\Functs(\BcondGal(X),\ICond(\Ani)) \arrow[r, "b^*"] & \Functs(\Gal(X),\ICond(\Ani)) \comma
		\end{tikzcd}
	\end{equation*}
compare \cite[Proof of Proposition 3.38]{The_condensed_homotopy_type}.
Further using the Grothendieck construction
$$ \CondAni_{/\CondShape{X}}=\Functs(\BcondGal(X),\ICond(\Ani))$$ proven in \cite[Corollary 3.20]{MR4574234}, the induced
functor $$\Psi^*\colon \CondAni_{/\CondShape{X}}\rightarrow \Xproethyp$$ is always fully faithful. Thus, statement (i) implies (ii). The other implication is also true as it follows by (ii) and the description of $\CondShape{X}$ as the relative shape over $\CondAni$, see \cref{rem:relative_shape}, that $\CondShape{X}=\ast.$
\end{proof}
We obtain a statement on condensed connectedness by taking the $0$-truncated version of the proposition.
\begin{corollary} Let $X$ be a qcqs scheme. The following assertions are equivalent.
\begin{itemize}
    \item[(i)] The scheme $X$ is condensed connected.
    \item[(ii)] The ($0$-truncated) pullback functor 
    $f^*_{\leq 0}\colon \CondSet \rightarrow \Sh(X_{\proet},\Set)$ is fully faithful.
\end{itemize}
\end{corollary}
\subsection{(Non-)examples of condensed contractible schemes}
Let us give (non-)examples of condensed connected and condensed contractible schemes.
\begin{notation}
For a field $k$ with a choice of a separable closure $k^{\mathrm{sep}}$, we denote by $\Gal_k$ the absolute Galois group of $k$ with respect to $k^{\mathrm{sep}}$. 
\end{notation}
\begin{recollection}
Every group $G\in \Grp$ can be regarded as a (connected) groupoid via its \textit{delooping} $\Bup G$ which is the groupoid with a single object $\ast$ and morphisms $\Hom_{\Bup G}(\ast,\ast)=G$. For a condensed group $G\in \CondGrp$, there exists likewise a \textit{condensed delooping} or \textit{condensed classifying anima} $\Bup G\in \CondAni$. It is given by taking level-wise the deloopings $\Bup G(S)$ of the groups $G(S)$ for all $S\in \Extr$. This defines a (1-)truncated condensed anima as delooping is compatible with products of groups, compare \cref{rec:adjointcondensed}.
\end{recollection}
\begin{example}{\cite[Example 3.41]{The_condensed_homotopy_type}}\label[example]{ex:condshape_field}
For every field $k$ and a choice of a separable closure $k^{\mathrm{sep}}$, we have $$\CondShape{\Spec(k)}=\BGal_k\in \CondAni$$ where $\Bup\Gal_k$ is the condensed delooping of the profinite absolute Galois group $\Gal_k$ of $k$ viewed as condensed group, compare \cref{ex:top}.
This object is always connected. It is contractible if and only if $k=k^{\mathrm{sep}}$.
\end{example}
\begin{lemma}{\cites[Corollary 3.42]{The_condensed_homotopy_type}[Lemma 2.4.5]{CatrinsThesis}}
Let $X$ be a scheme of dimension $0$. 
Then the condensed homotopy type is contractible $$\CondShape{X}=\ast$$ if and only if the reduced scheme $X_{\mathrm{red}}$ is $\Spec(k)$ for $k$ a separably closed field.
\end{lemma}
We extend the example of (separably closed) fields to (strictly henselian) local rings. As an upshot, we keep triviality on condensed connected components for local rings.
\begin{example}
Let $A$ be a local ring. Then $X=\Spec(A)$ is condensed connected: By \cref{prop:alter_Gal_prop}, for every $S\in \Extr$ the category
$\Gal(X)(S)$ admits a weakly initial object. Particularly, for all $S\in \Extr$ these categories are connected and thus their classifying anima are as well. 
\end{example}
Adding the henselian property in the previous example, all higher homotopy groups vanish. However, if the local rings are not strictly henselian, we have non-trivial automorphisms of points inducing homotopical information. 
\begin{example}\label[example]{ex:henselianring}
Let $X=\Spec(A)$ for $A$ a henselian local ring with residue field $\kappa$. By \cite[Corollary 3.48]{The_condensed_homotopy_type}, there is an equivalence
$$\CondShape{X}=\CondShape{\Spec(\kappa)}=\Bup\Gal_{\kappa}\period$$
The scheme $X$ is condensed contractible if and only if $A$ is strictly henselian.
\end{example}
In \cref{sec:weaklycontractible}, we introduced affine schemes which are highly determined by their local rings at closed points and their profinite sets of connected components. From our considerations on (strictly henselian) local rings, we can move on to \textit{w-(strictly) local} rings.
\begin{proposition}\label[proposition]{ex:wlocalring}
Let $X=\Spec(A)$ for $a$ a \textit{w}-local ring. Then it is $\pizerocond(X)=\pi_0(X)$ and the fiber of $\CondShape{X}\rightarrow \pi_0(X)$ at a point in $\pi_0(X)$ is the condensed homotopy type of the localization $X_{(x_0)}\colonequals\Spec(A_{(x_0)})$ at the closed point $x_0$ of the given connected component. 
\end{proposition}
\begin{proof}
  As in a $w$-local ring every prime ideal lies in a unique maximal ideal, we obtain the identification on connected components by \cite[Proposition 4.29]{The_condensed_homotopy_type} and under notice of the isomorphisms of topological spaces $$\mathrm{M}\Spec(A)\simeq\Spec(A)_{\mathrm{cl}}\simeq \pi_0(\Spec(A))\comma $$ where $\mathrm{M}\Spec(A)\subset \Spec(A)$ is the subspace of maximal ideals. The given map $$\CondShape{X} \rightarrow \pizerocond(X)=\pi_0(X)$$ is the unit of the adjunction between $\CondSet$ and $\CondAni$.
As every connected component in $\pi_0(X)$ corresponds to a unique closed point $x_0\in \Spec(A)_{\mathrm{cl}}$, we are aiming to show that the square 
    \[
  \begin{tikzcd}
   \CondShape{X_{(x_0)}}  \arrow[swap]{r} \arrow{d} & \CondShape{X} \arrow{d} \\
     \ast \arrow["x_0"]{r} & \pi_0(X)
  \end{tikzcd}
    \]
is a pullback square in $\CondAni$, where $\pi_0(X)\in \ProFin\subset \CondSet \subset \CondAni$ and $x_0\colon \ast \rightarrow \pi_0(X)$ maps to the connected component of $x_0$. As the localization functor $\Bcond\colon \CondCat \rightarrow \CondAni$ is locally cartesian, see \cite[Example 5.5]{The_condensed_homotopy_type}, we are reduced to showing that 
 \[
  \begin{tikzcd}
   \Gal(X_{(x_0)}) \arrow[swap]{r} \arrow{d} & \Gal(X)\arrow{d} \\
     \ast \arrow["x_0"]{r} & \pi_0(X)
  \end{tikzcd}
    \]
is a pullback square in $\CondCat$, where we view $\pi_0(X)\in \CondCat$ via the fully faithful embedding $\CondAni\hookrightarrow \CondCat$. We note that all condensed categories in the square lie in the image of the functor extending the embedding $\ProAnifin\hookrightarrow \Catinfty$ to internal categories
$$\iota \colon \Cat(\ProAnifin)\hookrightarrow \CondCat\comma$$
which is a fully faithful right adjoint, see \cite[Observation 2.21]{The_condensed_homotopy_type}. For the Galois category of a scheme, this is stated in \cite[Recollection 3.32]{The_condensed_homotopy_type} and for $\pi_0(X)$ it follows by the composed embedding $\ProFin \hookrightarrow \ProAnifin\hookrightarrow \Cat(\ProAnifin)$. Thus, by conservativity and right adjointness of the composed functor
$$\Cat(\ProAnifin) \xhookrightarrow{\iota} \CondCat \xrightarrow{X\mapsto X(\ast)} \Catinfty \comma$$
compare \cite[Observation 6.5]{The_condensed_homotopy_type}, we are further reduced to showing that the square
 \[
  \begin{tikzcd}
   \Pt({X_{(x_0)}}_{\et}) \arrow[swap]{r} \arrow{d} & \Pt(X_{\et})\arrow{d} \\
     \ast \arrow["x_0"]{r} & \pi_0(X)
  \end{tikzcd}
    \]
on underlying categories is a pullback square. By the description of $\Pt(X_{\et})$ as the category of strictly henselian local rings lying above $X$, \cref{rem:points_strictlyhenselian}, the pullback describes exactly those strictly henselian local rings $B$ such that the structure map $\Spec(B) \rightarrow X$ factors over the connected component of the closed point $x_0$. As this connected component is identified with the subscheme $X_{(x_0)}\subset X$, see \cref{rem:localring_connectedcomp}, the pullback is indeed $ \Pt({X_{(x_0)}}_{\et})$ and we are done.
\end{proof}
\begin{remark}
Let $A$ be a $w$-local ring with all local rings at closed points $x_0\in \Spec(A)$ henselian.
Combining \cref{ex:henselianring} and \cref{ex:wlocalring}, all the fibers from the previous proposition are $1$-truncated as they coincide with $\Bup\Gal_{\kappa(x_0)}\in \CondAni$, where $\kappa(x_0)$ is the residue field of the point $x_0$. They are $0$-truncated if and only if the local rings are even strictly henselian. This, in particular, implies that the whole condensed homotopy type is $1$-truncated. respectively, $0$-truncated as all (higher) condensed homotopy groups at points $\ast \rightarrow \CondShape{\Spec(A)}$ then vanish.
\end{remark}
In the strictly henselian case, the final point of the last remark was already stated in \cite{The_condensed_homotopy_type}:
\begin{example}\label[example]{ex:strictly_henselian_trivial}
Let $X$ be the spectrum of a $w$-strictly local ring.
By \cite[Proposition 3.43]{The_condensed_homotopy_type}, the condensed homotopy type is $0$-truncated with $\CondShape{X}=\pi_0(X)$.
Then $X$ is condensed contractible if and only if $X$ is connected, i.e., $X$
is the spectrum of a strictly henselian local ring.
\end{example}
Later in \cref{sec:Trivial_Shapes}, we will see that $A$ being strictly henselian is equivalent to $\Gal(\Spec(A))$ having an initial object. We will show a corresponding result on $\Gal(X)$ having a terminal object for $X$ being irreducible and everywhere strictly local. Here, the irreducibility ensures the scheme to be condensed connected.
\begin{example}
Every irreducible qcqs scheme $X$ is condensed connected. This follows as by \cite[Corollary 4.19]{The_condensed_homotopy_type} or \cite[Proposition 2.2.25]{CatrinsThesis}, the condensed set of connected components $\pizerocond(X)$ coincides in this case with the usual profinite set of connected components $\pi_0(X)$ and irreducibility of a topological space induces connectedness.
\end{example}
\begin{recollection}
We say that a scheme $X$ is \textit{\'etale contractible} if its \textit{protruncated \'etale homotopy type} $\Pietprotrun(X)\in \ProAnitrun$, defined as the \textit{protruncated shape} of the (hypercomplete) \'etale $\infty$-topos\footnote{For a definition of (protruncated) shapes see, for example, \cite[Recollection 2.27]{The_condensed_homotopy_type}.}, is a contractible pro-anima. Equivalently, this is the case if $X$ is connected and all its \'etale homotopy pro-groups $\pi_n^{\et}(X)$ are contractible, see \cites[Remark 4.1.3]{Exodromy}[Corollary 4.4]{MR245577}.
\end{recollection}
\begin{remark}
Every condensed contractible scheme also is \'etale contractible.
This follows by \cref{rec:recover_etale_homotopy}: The prodiscretization functor $(-)\prodisccompl\colon\CondAni \rightarrow \ProAnitrun$ preserves the terminal object as it is left adjoint to the pro-extension of the inclusion $\Ani_{<\infty}\hookrightarrow \CondAni$.
\end{remark}
\begin{example}
In \cite[Theorem 1]{MR3248993}, Holschbach-Schmidt-Stix prove that there does not exist a non-trivial smooth variety in positive characteristic which is \'etale contractible.
By the previous remark, this observation likewise translates to condensed contractibility.
\end{example}
Studying the previous examples, one might wonder whether the condensed homotopy type adds anything new to the picture of homotopy theory for schemes. The next example on a scheme that is \'etale but not condensed contractible might give an answer. 
\begin{example}\label[example]{ex:affineline_not_condensedcontractible}
In \cite[Section 7.1]{The_condensed_homotopy_type}, it is proven that the condensed fundamental group of the affine line over the complex numbers $\pionecond(\AA^1_{\CC}, \etabar)$ is non-trivial. Hence, $\AA^1_{\CC}$ is not condensed contractible although it is \'etale contractible: By the profiniteness theorem \cite[Theorem 11.1]{MR245577}, the \'etale homotopy type of $\AA^1_{\CC}$ is profinite and coincides by the generalized Riemann existence theorem \cite[Theorem 12.9, Corollary 12.10]{MR245577} with the profinite completion of the classical homotopy type $\Shape(\CC)$ of the complex numbers which is contractible.
\end{example}
The following section deals with an extension of this example and especially a generalization of the computation of $\pionecond(\AA^1_{\CC}, \etabar)$ given in \cite[Section 7.1]{The_condensed_homotopy_type} to general Dedekind domains.
\subsection{Condensed fundamental groups of Dedekind domains}
In the following, we give a formula for the computation of the underlying group of the condensed fundamental group of a general Dedekind domain. Following an argument by Jakob Stix, we show that for $\Spec(\ZZ)$ the resulting quotient and its abelianization are non-trivial. In other words, we show that $\Spec(\ZZ)$ provides another example of an \'etale contractible scheme, compare \cite[Example 11.2.2]{Exodromy}, that is not condensed contractible.
\\
\newline
We recollect notation from the theory of valued and local fields. Consult, e.g., \cite{valuedfields} and \cite{Localfields} for more details.
\begin{recollection}
Let $A$ be a Dedekind domain with fraction field $K=\Frac(A)$. 
\begin{enumerate}
     \item We fix a separable closure $K^{\mathrm{sep}}$ and an algebraic closure $\overline{K}$ of the field $K$ together with embeddings $K\subset K^{\mathrm{sep}}\subset \overline{K}.$ We denote by $\Gal_K\colonequals \Gal(K^{\mathrm{sep}},K)$ the absolute Galois group of $K$ with respect to $ K^{\mathrm{sep}}$.
    \item For $p$ a non-zero prime ideal of $A$, we denote by $v_p$ the \textit{$p$-adic valuation} of $K$. The valuation ring $\mathcal{O}_{v_p}$ of the valued field $(K,v_p)$ is the localization of $A$ at $p$, denoted by $A_p$.
    We also write $(K,\mathcal{O}_{v_p})$ or simply $(K,A_p)$ for the valued field. The residue field at the prime $p$ is denoted by $\kappa(p)\colonequals A_p/pA_p$ and its Galois group with respect to a fixed separable closure $\kappa(p)^{\mathrm{sep}}$ by $\Gal_{\kappa(p)}\colonequals \Gal(\kappa(p)^{\mathrm{sep}}, \kappa(p)).$
    \item For every non-zero prime ideal $p\subset A$, we fix a \textit{prolongation} (or extension) 
    $w_p$ of the $p$-adic valuation $v_p$ in the separable closure $K^{\mathrm{sep}}$ with corresponding valuation ring $\mathcal{O}_{w_p}$, i.e., the valuation rings satisfy $\mathcal{O}_{w_p}\cap K=A_p$. In other words, we fix a prime ideal $\beta$ above $p$ in the integral closure ${O}_{K^{\mathrm{sep}}}$ of $A$ in $K^{\mathrm{sep}}$, i.e., it is $\beta \cap A = p$ or, equivalently, $\beta$ contains the ideal $p{O}_{K^{\mathrm{sep}}}$ generated by $p$. The prolongation $w_p$ and the prime ideal $\beta$ correspond to each other in the sense of $v_{\beta}=w_p$ and $\beta$ is the maximal ideal of the local ring $\mathcal{O}_{w_p}$.  We denote the \textit{decomposition group} of $w_p$ (or $\beta$) by 
    $$D_p\colonequals \{\sigma \in \Gal_K| \sigma(\mathcal{O}_{w_p})=\mathcal{O}_{w_p}\}= \{\sigma \in \Gal_K| \sigma(\beta)=\beta\}\subset \Gal_K$$ and its \textit{inertia group} by 
    \begin{align*}
    I_p&\colonequals \{\sigma \in \Gal_K| \forall x \in \mathcal{O}_{w_p}:  \sigma(x)\equiv x\mod \beta\}\\
    &\simeq \ker(D_p \rightarrow \Gal_{\kappa(p)}) \subset D_p\comma    
    \end{align*}
    where the surjective morphism $D_p \rightarrow \Gal_{\kappa(p)}$ is defined by passage to the quotient. 
\end{enumerate}
\end{recollection}
\begin{remark}
    Note that even though we make choices by fixing a separable and an algebraic closure as well as an extension of the given valuation (ring), our subsequent results are not affected in a serious way by these choices: The groups $\Gal_{\kappa(p)}, \Gal_K, D_p$ and $I_p$ depending on these choices are conjugate to the corresponding groups for a different choice, compare, e.g., \cite[Lemma 5.2.1]{valuedfields} for the according statement on decomposition groups.
\end{remark}
\begin{recollection}\label[recollection]{rec:notation_dedekind}
    Let $X=\Spec(A)$ be the associated affine scheme of a Dedekind domain $A$.
    As $X$ is irreducible by $A$ being an integral domain, it is condensed connected, i.e., the condensed fundamental groups with respect to different basepoints are isomorphic.
    For the generic point $\eta \in X$ (corresponding to the zero ideal of $A$), we fix a geometric point $\etabar \rightarrow X$, i.e., a morphism $\Spec(K^{\mathrm{sep}})\rightarrow \Spec(A)$ corresponding to the embedding $K\hookrightarrow K^{\mathrm{sep}}$ and mapping to $\eta\in\Spec(A)$.
\end{recollection}
\begin{theorem}\label[theorem]{prop:Dedekind_domain}
Let $X=\Spec(A)$ be the affine scheme of a Dedekind domain $A$ with fraction field $K=\Frac(A)$ and generic point $\etabar \rightarrow X$. For every non-zero prime ideal $p$, we have an (up to conjugation unique) inertia group $I_p\subset \Gal_K$, chosen as above. Let $N$ be the (abstract) normal closure of the subgroup of $\Gal_K$ generated by the images of the inertia groups $I_p$ at all non-empty primes $p$.
Then the underlying group of the condensed fundamental group of $X$ is the quotient of $\Gal_K$ by $N$, i.e., 
$$\pionecond(X,\etabar)(\ast)=\Gal_K/N\period$$
\end{theorem}
\begin{proof}
First, we would like to point out that the proof uses various aspects on stratified spaces and $\infty$-topoi formulated in detail in \cite{Exodromy} as well as more classical results on local and global field theory, which we will not give a proper introduction into here as this would go beyond the scope of this article. Nevertheless, we made an effort to make the arguments accessible.
    We generalize the computation in the proof of \cite[Proposition 7.2]{The_condensed_homotopy_type} on non-triviality of the condensed fundamental group $\pionecond(\AA^1_{\CC}, \etabar)$.
    By \cref{rec:notation_dedekind}, the underlying anima $ (\Bcond\Gal(X))(\ast)=\BGal(X)(\ast)\in \Ani $ of the condensed classifying anima is connected.
    We can compute it using 
    the natural equivalence
    \begin{equation}\label{eq:classifying_anima_via_decollage}
        \BGal(X)(\ast) \equivalent \colim_{\subdiv(\Zpos{X})^{\op}} \Nup_{\Zpos{X}}(\Gal(X)(\pt))
    \end{equation}
given in \cite[Example 7.6]{The_condensed_homotopy_type}. This equivalence arises from the fact that the underlying category $\Gal(X)(\ast)$ is \textit{stratified} over the Zariski poset $\Zpos{X}$, i.e., it comes with a conservative functor $\Gal(X)(\ast) \rightarrow \Zpos{X}$. This map was already described in \cref{ex:Gal(X)_is_a_stratified_anima}. An $\infty$-category $\Ccal$  with a conservative functor $\Ccal \rightarrow P$ to a poset $P$ is called a \textit{$P$-stratified space}. The stratification decomposes the space in \textit{strata} corresponding to points $x_0\in P$ and a \textit{link} for each pair of points with $x_0 \leq y_0$.
The \textit{nerve} $\Nup_P(\Ccal)$ of a $P$ -stratified space $\Ccal \rightarrow P$ keeps track on the strata and the linking information. The assignment $\Ccal \mapsto \Nup_P(\Ccal)$ defines a functor from $P$-stratified spaces to \textit{spatial d\'ecollages} over $P$, which is an equivalence of categories \cite[Theorem 2.7.4]{Exodromy}.
The spatial d\'ecollage $\Nup_P(\Ccal)$ is given as the functor \begin{equation*}
       \Nup_{P}(\Ccal) \colon \subdiv(P)^{\op} \to \Ani, \ \Sigma \mapsto \Hom_{P}(\Sigma,\Ccal),
    \end{equation*}
see \cite[Construction 2.71]{Exodromy}.
Here, the subdivision poset $\subdiv(P)$ consists of nonempty linearly ordered finite subsets of $P$ and is ordered by containment, see \cite[Notation 1.1.8]{Exodromy}. For the $\infty$-category $\Ccal$, the classifying anima $B\Ccal$,  obtained by inverting everything, can be computed as the colimit $\colim_{\Ccal} 1_{\Ani}$ of the constant diagram $\Ccal \rightarrow \Ani$ at the terminal object, compare \cite[Notation 2.2.1]{exodromy}. Under the equivalence of $P$-stratified spaces and spatial d\'ecollages over $P$, this colimit translates to 
taking the colimit over the nerve $\Nup_P(\Ccal)$ indexed by the subdivision poset $\subdiv(P)$. This was made formal in \cite[Proposition 7.5]{The_condensed_homotopy_type} providing us with (\ref{eq:classifying_anima_via_decollage}).
Coming back to our concrete application case, by combining \cite[Construction 8.71, Example 8.72 and Theorem 10.4.3]{Exodromy}, the nerve $\Nup_{\Zpos{X}}(\Gal(X)(\pt))$
can be described as follows: First take the nerve  $$ \Nup_{\Zpos{X}}(X_{\et})\colon\subdiv(\Zpos{X})^{\op} \to \RTop_{\infty} $$ of the stratified $\infty$-topos $X_{\et}\rightarrow X_{\zar}$, then take the profinite shape $\Pietprofin\colon \RTop_{\infty} \rightarrow \ProAnifin$ and finally apply the materialization functor on profinite anima $ \lim \colon \ProAnifin \to \Ani $.
As the Zariski poset $\Zpos{X}$ consists of the generic point $\eta$ and points $p$ for all finite primes in $X$ and it is $p<\eta$, as $p$ lies in the closure of $\eta$, the elements of the subdivision poset $\subdiv(\Zpos{X})$ are of the form
    \begin{equation*}
        \{p\} \comma \quad \{\eta\} \comma \andeq \{p  < \eta \} \period
    \end{equation*}
The ordering is given by $ \{p\} < \{p  < \eta \}$ and $\{ \eta \}< \{p  < \eta \}$. The nerve functor on stratified $\infty$-topoi $\Nup_{\Zpos{X}}(X_{\et})\colon\subdiv(\Zpos{X})^{\op} \to \RTop_{\infty} $ sends $\{p\}$ and $\{\eta\}$ to the topos-theoretic strata $(X_{\et})_p$, respectively $(X_{\et})_{\eta}$, and $\{p < \eta\}$ to the topos-theoretic link given by the oriented fiber product $(X_{\et})_p \orientedtimes_{X_{\et}}(X_{\et})_{\eta}$, see \cite[Construction 8.7.1]{Exodromy}.
By \cite[Example 9.2.4]{Exodromy}, the $p$-strata are given by the \'etale $\infty$-topoi of the residue fields $\kappa(p)$ and the $\eta$-stratum is given by the \'etale topos of the residue field $\kappa(\eta)=K$.
Further using \cite[Example 5.4.10]{Exodromy}, the link $(X_{\et})_p \orientedtimes_{X_{\et}} (X_{\et})_{\eta}$ can be computed as $$((X_{\et})_p \orientedtimes_{X_{\et}} X_{\et}) \times_{X_{\et}} (X_{\et})_{\eta}\comma$$ where we take the fiber product in $\infty$-topoi.
The oriented fiber product $((X_{\et})_p \orientedtimes_{X_{\et}} X_{\et})$ is by \cite[Example 6.7.4]{Exodromy} the \'etale $\infty$-topos corresponding to the following ring: take the hensilization $A_p^h$ of the local ring $A_p$ at the prime $p$ and then the unramified extension which has the same residue field $\kappa(p)$. By the condition on the residue field, this has to be the ring $A_p^h$ itself.
Thus we are reduced to compute the fiber product 
$$\Spec(A_p^h)_{\et}\times_{X_{\et}} \Spec(K)_{\et}$$
of \'etale $\infty$-topoi. By \cite[Lemma 2.9]{reconstruction_etale_topoi}, this is the same as the \'etale $\infty$-topos of the fiber product on the level of schemes, i.e., of $\Spec(A_p^h)\times_{X}\Spec(K)$, which again, by \cite[\href{https://stacks.math.columbia.edu/tag/01I4}{Lemma 01I4}]{stacksproject} corresponds to the tensor product of rings $A_p^h \otimes_{A}K=\Frac(A_p^h)$.
Altogether, any diagram $\{p\} < \{p < \eta\} < \{\eta\}$ in the subdivision poset is sent to the span of schemes
\begin{equation}\label{eq:DVR_decollage1}
        \begin{tikzcd}[sep=3em]
        	\Spec(\kappa(p)) & \Spec(\Frac(A_p^h)) \arrow[r] \arrow[l] & {\Spec(K)}
        \end{tikzcd}
    \end{equation} 
before applying the profinite shape (of the corresponding \'etale $\infty$-topoi) and the materialization itemwise.
The field $\Frac(A_p^h)$ is an algebraic field extension of $K$ as the ring $A_p^h$ is a filtered colimit of finite \'etale $A_p$-algebras each of which gives a finite separable field extension on fraction fields.
Thus we can choose a lift $\etabar_p $ of the geometric point $\etabar$ lying above the generic point $\eta \in \Spec(A)$, which fits into the commutative triangle
    \begin{equation*}
        \begin{tikzcd}
        	 & \Spec(\Frac(A_p^h))\arrow[d]  \\
        	\Spec(K^{\mathrm{sep}}) \arrow[r, "\etabar"'] \arrow[ur, dotted, "{\etabar_p}"] & \Spec(K) \period
        \end{tikzcd}
    \end{equation*}
The lift also induces a geometric point $\etabar_p'\colon \Spec(K^{\mathrm{sep}}) \rightarrow \Spec(\kappa(p))$.
In particular, we can lift the span \eqref{eq:DVR_decollage1} to a span of pointed schemes. Therefore, $\Nup_{\Zpos{X}}(\Gal(X)(\ast))$ also lifts to a diagram of pointed anima $ \Nup_{\Zpos{X}}(\Gal(X)(\ast))_*\colon\subdiv(\Zpos{X})^{\op} \to \Ani_*$.
    Using that $\uppi_1$ is an equivalence between pointed, connected, $1$-truncated anima and the category of groups \HTT{Proposition}{7.2.12}, we may thus compute
    \begin{equation*}
        \uppi_{1}(\Bup\Gal(X)(\ast),\etabar) \simeq \colim_{\subdiv(\Zpos{X})^{\op}} \uppi_{1}(\Nup_{\Zpos{X}}(\Gal(X)(\ast))_*) \period
    \end{equation*}
    
    Now for any $\{p\} < \{p  < \eta \} > \{ \eta \}$, the corresponding span in groups is given by
    \begin{equation*}
        \begin{tikzcd}[sep=3em]
        	\pioneet(\Spec(\kappa(p)),\etabar_p') & \pioneet(\Spec(\Frac(A_p^h),\etabar_p) \arrow[r] \arrow[l] & \pioneet(\Spec(K,\etabar) \period
        \end{tikzcd}
    \end{equation*}
    Let $K_p$ be the completion of $K$ with respect to the $p$-adic valuation $v_p$. The absolute Galois group of this field is the decomposition group $\Gal_{K_p}=D_p$ \cite[Chapter II, §3, Corollary 4]{Localfields} and hence a closed subgroup of $\Gal_K$ \cite[Lemma 5.2.1]{valuedfields}. As such, by the infinite version of the fundamental theorem of Galois theory \cite[Theorem 5.1.2]{valuedfields}, it equivalently is the Galois group of its fixed field over $K$, which is the henselization $(K^h,A_p\cap K^h)$ of the valued field $(K,A_p)$ \cite[Theorem 5.2.2 and before]{valuedfields}. Note that it identifies $K^h=\Frac(A_p^h)$ by the universal property of henselization, cf. \cite[Theorem 5.2.2 (2)]{valuedfields}. Thus we can identify $\Gal_{\Frac(A_p^h)}=\Gal_{K_p}$.
    Translating the \'etale fundamental groups from above to absolute Galois groups, we obtain
     \begin{equation*}
        \begin{tikzcd}[sep=3em]
        	\Gal_{\kappa(p)} & \Gal_{K_p} \arrow[r] \arrow[l] & \Gal_{K} \period
        \end{tikzcd}
    \end{equation*}
Further, $\Gal_{\kappa(p)}$ is isomorphic to the quotient of the decomposition group $D_p=\Gal_{K_p}$ by its normal inertia subgroup $I_p$ \cite[Chapter I, §7, Proposition 20]{Localfields}.
Thus the span translates to
 \begin{equation*}
        \begin{tikzcd}[sep=3em]
        	D_p/I_p & D_p\arrow[r] \arrow[l] & \Gal_{K} \period
        \end{tikzcd}
    \end{equation*}
Then the colimit of the diagram $ \uppi_{1}(\Nup_{\Zpos{X}}(\Gal(X)(\ast))_*) $ over $\subdiv(\Zpos{X})^{\op}$ computes as the quotient of $\Gal_{K}$ by the (abstract) normal closure of the subgroup generated by the images of all the inertia groups $I_p$. 
We refer to this normal closure as $N$.
Indeed, as the kernels of the maps $\Gal_{K_p}\rightarrow \Gal_{\kappa(p)}$ are exactly the inertia groups $I_p$, the kernel of the map from $\Gal_{K}$ to the colimit is exactly generated by all these groups.
\end{proof}
\begin{example}
    For $X=\AA^1_{\CC}=\Spec(\CC[T])$ the affine line over the complex numbers, the residue field $\kappa(p)$ at a non-trivial prime ideal $p=(T-a)\subset \CC[T]$ is algebraically closed, namely, it is isomorphic to $\CC$. Thus all the quotients $D_p/I_p$ are trivial or, equivalently, every inertia group $I_p$ already is all of the decomposition group $D_p$. By the last proposition, the group $\pionecond(X,\etabar)(\ast)$ is then the quotient of $\Gal_{\CC(T)}$ by the abstract normal closure of the subgroup generated by the decomposition groups, which was shown to be non-trivial in \cite[Corollary 7.8]{The_condensed_homotopy_type}.
\end{example}
\begin{example}\label[example]{ex:fundamental_group_integers}
Let $X=\Spec(\ZZ)$ and $\etabar \rightarrow \Spec(\ZZ)$ be a geometric point corresponding to the generic point $\eta \in \Spec(\ZZ)$. We argue that the condensed fundamental group $\pionecond(\Spec(\ZZ), \etabar)$ is non-trivial although the \'etale fundamental group is trivial, see \cite[Example 11.2.2]{Exodromy}.
By \cref{prop:Dedekind_domain}, we need to show that the quotient of $\Gal_{\QQ}$ by the \textit{abstract} normal closure $N$ of the subgroup generated by the inertia groups $I_p$ is non-trivial.\footnote{Note that the quotient by the \textit{topological} normal closure of the subgroup generated by all inertia groups is known to be trivial: By infinite Galois theory, this closed subgroup corresponds to a field extension of $\QQ$ that has to be unramified at all $p$.
But there is no such non-trivial extension of $\QQ$, 
see \cite[Chapter III, §2, Theorem 2.18]{Algebraicnumbertheory}.}
We learned the following argument from Jakob Stix: 
We consider the abelianization map 
\begin{align}\label{eq:abelianization}
  \Gal_{\QQ} \rightarrow \Gal_{\QQ}^{\ab}  
\end{align}
on topological groups, i.e., $\Gal_{\QQ}^{\ab}$ is the quotient of $\Gal_{\QQ}$ by the closure of the (group-theoretical) commutator group.
The topological group $\Gal_{\QQ}^{\ab}$ is the Galois group of the maximal abelian extension $\QQ^{\mathrm{ab}}$ of $\QQ$, i.e., the union over all finite abelian extensions in the separable closure. Further, we have $$\Gal_{\QQ}^{\ab}\overset{(1)}{\simeq} \Gal(\QQ(\mu_{\infty})/\QQ) \overset{(2)}{\simeq} \hat{\ZZ}^{\times}=\prod_p \ZZ_p^{\times}\period$$
Here the identification $(1)$ follows as by the (global) Kronecker-Weber theorem (e.g., \cite[Chapter V, §1, Theorem 1.10]{Algebraicnumbertheory}) the maximal abelian extension $\QQ^{\mathrm{ab}}$ of $\QQ$ is the maximal cyclotomic extension $\QQ(\mu_{\infty})$ of $\QQ$, i.e., the field  obtained by adjoining all roots of unity.
The Galois group of this extension is the cofiltered limit, indexed by the natural numbers $\NN$, over the Galois groups $\Gal(\QQ(\mu_{n})/\QQ)=(\ZZ/n\ZZ)^{\times}$ where $\mu_n$ is the group of all $n$-th roots of unity, compare \cite[Chapter IV, §2, Example 6]{Algebraicnumbertheory}. This gives the identification $(2)$.
For every non-zero prime $p$, the image of every inertia group $I_p\subset \Gal_{\QQ}$ under the abelianization map (\ref{eq:abelianization}) is an inertia group $I_p^{\mathrm{ab}}$ of $\QQ^{\mathrm{ab}}$ which is given as follows:
We know that the decomposition group $D_p\subset \Gal_{\QQ}$ is the absolute Galois group of the $p$-adic completion $\QQ_p$ by~\cite[Chapter II, §3, Corollary~4]{Localfields}.
Its image under (\ref{eq:abelianization}) is the Galois group $\Gal_{\QQ_p}^{\mathrm{ab}}=\Gal(\QQ_p^{\mathrm{ab}}, \QQ_p)$ of the maximal abelian extension of $\QQ_p$.
The \textit{reciprocity map} $\QQ_p^{\times} \rightarrow \Gal_{\QQ_p}^{\mathrm{ab}}$ then identifies $\ZZ_p^{\times}\subset \QQ_p^{\times}$ with the abelian inertia subgroup $I_p^{\mathrm{ab}}\subset\Gal_{\QQ_p}^{\mathrm{ab}}$ by \cite[Lemma 9.8]{classfieldtheory}. As in abelian groups every subgroup is normal, the image of the normal subgroup $N$ under (\ref{eq:abelianization}) is the subgroup $\bigoplus_p \ZZ_p^{\times}\subseteq \hat{\ZZ}^\times$ generated by the abelian inertia groups $I_p^{\ab}=\ZZ_p^{\times}$. As it is a proper subgroup, we deduce that the abelianization $(\Gal_{\QQ}/N)^{\ab}\simeq\hat{\ZZ}^{\times}/\bigoplus_p \ZZ_p^{\times}$ is non-trivial. This implies also the quotient $\Gal_{\QQ}/N$ to be non-trivial.
\end{example}
\begin{remark}
We learned from Marcin Lara that $\Gal_{\QQ}/N\neq 1$ also follows by an argument on cyclotomic characters which shows that the subgroup $N$ cannot contain complex conjugation. Indeed, every $g\in N$ can be written as a finite word $g=g_1g_2\cdot\dots\cdot g_n$ where each $g_i$ is in the normalization of some inertia group $I_{p_i}$. If we pick a prime $\ell$ larger than any of the $p_i$ appearing in the presentation of $g$, then evaluation of the cyclotomic character $\chi_{\ell}\colon G_{\QQ}\rightarrow \ZZ_{\ell}^{\times}$ on $g$ is trivial in $\ZZ_{\ell}^{\times}$. Indeed, the cyclotomic character $\chi_{\ell}$ is defined by the action of $G_{\QQ}$ on all $l$-power roots of unity $\zeta_{\ell
^n}$. As for all $n \geq 1$ the extension $\QQ(\zeta_{\ell
^n})/\QQ$ is unramified at all $p\neq \ell$, see \cite[Chapter I, Corollary 10.4]{Algebraicnumbertheory}, the restrictions of the inertia subgroups $I_{p_i}\subset G_{\Q}$ to  $\QQ(\zeta_{\ell
^n})$ are trivial and thus the $I_{p_i}$ act trivially on all $\zeta_{\ell
^n}$. For the complex conjugation $\tau\in \Gal_{\QQ}$ we have $\chi_{\ell}(\tau)=-1\in \ZZ_{\ell}^{\times}$ for every cyclotomic character $\chi_{\ell}\colon \Gal_{\QQ}\rightarrow \ZZ_{\ell}^{\times}$. This is true as every root of unity is sent to its inverse under complex conjugation. By this, the complex conjugation cannot be presented as such a $g$ and thus is not contained in $N$.
\end{remark}
\section{A classification of condensed contractible schemes}\label[section]{sec:Trivial_Shapes}
We can prove properties of the condensed homotopy type $\CondShape{X}$ by studying $\Gal(X)$ and deduce from there.
Here we want to give a partial classification of \textit{condensed contractible} schemes via properties of their Galois categories forcing the condensed homotopy type to become trivial.
\begin{remark}
Our use of the term \textit{trivial} is not totally well-suited. 
One could widen it to the case that the condensed homotopy type corresponds to a discrete collection of points.
For example, this happens if the given scheme is a disjoint union of condensed contractible schemes in the sense defined previously.
Thus, in the subsequent discussion, one should have \textit{condensed connected} schemes in mind, i.e., schemes $X$ such that $\pizerocond(\CondShape{X})=\ast$.    
\end{remark}
\subsection{Galois categories with trivial classifying anima}
\begin{definition}
We say that a condensed $\infty$-category $\Ccal\in \CondCat$ 
has a \textit{terminal}, respectively, an \textit{initial} object or is \textit{(co)filtered}
if this is true level-wise for all of the $\infty$-categories $\Ccal(K)$ with $K\in \Extr$.
\end{definition}
\begin{corollary}
Let $\Ccal \in \CondCat$ be a condensed $\infty$-category.
Then, the condensed classifying anima $\Bup\Ccal$ of $\Ccal$ is contractible if $\Ccal$ has a terminal object, respectively, an initial object, or, more generally, is (co)filtered.
\end{corollary}
\begin{proof}
This follows directly from the corresponding results on $\infty$-categories:
An $\infty$-category with an initial or terminal object is weakly contractible by \cite[\href{https://kerodon.net/tag/02P2}{Corollary 02P2}]{Kerodon}. The same holds for the (co)filtered case \cite[\href{https://kerodon.net/tag/02PH}{Proposition 02PH}]{Kerodon}.
Weakly contractible exactly means that the classifying anima becomes contractible by definition \cite[\href{https://kerodon.net/tag/00UA}{Definition 00UA}]{Kerodon} and left adjointness of the functor $\Bup\colon \Catinfty \rightarrow \Ani$.
\end{proof}
We are going to study conditions we must place on a scheme $X$ in order to apply this corollary to the Galois category $\Gal(X)$.
Along the way, we will see that for $\Gal(X)\in \Cond(\Ccal)$, the notions of terminal and filtered, respectively, initial and cofiltered collapse:
\begin{proposition}\label[proposition]{prop:filt_term_ini} Let $X$ be a qcqs scheme. Then, the following is always true:
\begin{itemize}
    \item[a)] $\Gal(X)$ has a terminal object $\iff$ $\Gal(X)$ is filtered.
    \item[b)] $\Gal(X)$ has an initial object $\iff$ 
    $\Gal(X)$ is cofiltered.
\end{itemize}
\end{proposition}
As any category with a terminal object is filtered and any category with an initial object is cofiltered, the respective implications from the left to the right are clear.
On the level of the underlying category $\Gal(X)(\ast)=\Pt(X_{\et})$, the equivalence follows from corresponding results for the Zariski poset $ \Zpos{X}$. Recall the definition of the poset $ \Zpos{X} $ from \cref{rec:Zariskiposet}.
\subsubsection{\cref{prop:filt_term_ini} on underlying categories}\label[subsection]{subsec:underlying_cats}
We show the statement of \cref{prop:filt_term_ini} on the level of the underlying category $\Gal(X)(\ast)$.
While first reading, this subsection can be skipped. Continue with \cref{sec:Classification} on the classification of schemes whose Galois category has an initial or terminal object.
\begin{lemma}\label[lemma]{lem:terminal_initial_underlying_cats}
Let $X$ be a qcqs scheme.
Then the following is always true:
\begin{itemize}
    \item[a)] $\Pt(X_{\et})$ has a terminal object $\iff$ $\Pt(X_{\et})$ is filtered.
    \item[b)] $\Pt(X_{\et})$ has an initial object $\iff$ 
    $\Pt(X_{\et})$ is cofiltered.
\end{itemize}
\end{lemma}
Note that for posets, being (co)filtered is equivalent to being \textit{(co)directed} and terminal, respectively, initial objects correspond to \textit{maximal}, respectively, \textit{minimal} elements.
\begin{lemma}\label[lemma]{lem:XZarconditions} Let $X$ be a qcqs scheme.
The following assertions are equivalent.
\begin{itemize}
\item[a)] The Zariski poset $ \Zpos{X}$ has a maximal element (respectively, a minimal element).
\item[b)] The Zariski poset $ \Zpos{X}$ is (co)directed.
\item[c)] The scheme $X$ is irreducible (respectively, local, i.e., corresponds to a local ring).
\end{itemize} 
\end{lemma}
\begin{proof}
The equivalences of a), b) and c) rely on the definition of the order by specializations.
For a)$\implies b)$ note that a poset with a maximal or minimal object is, in particular, (co)directed. For  b) $\implies$ c), we have to argue why there is a unique generic (closed) point:
The poset $ \Zpos{X}$ to be (co)directed means that any two points are specializations (generalizations) of a common point.
We handle the two cases separately.
\begin{itemize}
    \item[\textbf{1.}] \textbf{The directed case:}
    To satisfy the condition on specializations, the sober topological space $|X|$ cannot have more than one irreducible component since two maximal points, i.e., the unique generic points of irreducible components, do not specialize further.
    This implies irreducibility.
    \item[\textbf{2.}] \textbf{The codirected case:}
    Every closed point of $|X|$ does not generalize further as there is no other point lying in its closure. Hence, a scheme with more than one closed point contradicts the condition on $ \Zpos{X}$ to be codirected.
    A scheme that is quasi-compact and contains exactly one closed point is known to be equivalent to the spectrum of a local ring.
    This follows from the fact that every quasi-compact scheme, and thus also every closed subset of a quasi-compact scheme, has a closed point.
    Then the complement of an affine open containing the unique closed point has to be empty.
    This already implies that the scheme itself is affine with a unique closed point corresponding to a unique maximal ideal of the underlying ring.
    \end{itemize}
From c) to a), it is clear that the unique generic point or, respectively, the unique closed point of the scheme constitutes a maximal or, respectively, a minimal element of $ \Zpos{X}$.
\end{proof}
\begin{remark}\label[remark]{rem:filtered_irreducible}
If the category of points $\Gal(X)(\ast)=\Pt(X_{\et})$ of a qcqs scheme $X$ is (co)filtered, the underlying Zariski poset $ \Zpos{X}$ is (co)directed.
By the previous lemma, the scheme already needs to be irreducible, respectively, local.     
\end{remark}
\begin{definition}
We call an object in an $\infty$-category \textit{weakly terminal} if all mapping spaces into it are non-empty. Similarly, a \textit{weakly initial} object is defined by all mapping spaces out of it to be non-empty.  
\end{definition}
This definition differs from the one of initial, respectively, terminal objects only by the fact that no uniqueness of morphisms (up to a contractible choice) is required.
\begin{lemma}\label[lemma]{lem:weaklyobjects}
Let $X$ be a qcqs scheme.
The category $\Gal(X)(\ast)$ has a weakly terminal (respectively, weakly initial) object if and only if $X$ is irreducible (respectively, local).\footnote{This statement is already mentioned as an equivalence for the profinite category version of $\Gal(X)$ in \cite[Corollary 12.3.5]{Exodromy}.
For sake of completeness, we include a proof here.}
\end{lemma}
\begin{proof}
Assume that $\Gal(X)(\ast)$ has a weakly terminal, respectively, a weakly initial, object.
By the description of $\Gal(X)(\ast)$ in \cref{rem:Pointsetaletopos} and \cref{rem:Pointsetaletopos2}, this means nothing else than that there exists a geometric point of $X$ such that every other point is a specialization, respectively, a generalization of this point.
In other words, the scheme $X$ has a generic point, respectively, a special point that it is contained in the closure of every other point.
As the closure of the irreducible set consisting of the generic point is irreducible again, the existence of a generic point already implies that $X$ is irreducible.
For locality, we use the same argument as in  \cref{lem:XZarconditions}.
We take $U$ an open affine containing the special point.
Its complement has to be empty as otherwise the special point would be contained in it, which contradicts the choice of $U$.
Thus the scheme itself is affine and the special point corresponds to a prime ideal containing all other prime ideals.
This already implies that the underlying ring has a unique maximal ideal, i.e., is a local ring.
\end{proof}
\begin{example}
In general, in neither of the two cases the "weakly" condition can be dropped: Consider the spectrum $\Spec(k)$ of $k$ a non separably closed field.
This scheme is both irreducible and local but its Galois category $$\Gal(X)\simeq\CondShape{X}\simeq \BGal_k$$ (see \cref{ex:condshape_field}) does not have a (strictly) terminal or initial object.
\end{example}
\begin{remark}\label[remark]{rem:morphism_empty}
A weakly terminal or a weakly initial object of $\Gal(X)(\ast)$ carries an additional property: All morphism sets out of the weakly terminal and into the weakly initial object are non-empty if and only if they belong to the automorphism set of the object, compare \cref{rem:Pointsetaletopos2}.
\end{remark}
The notions of weakly terminal and weakly initial can be extended to condensed $\infty$-categories in a canonical way by requiring it to hold level-wise.
From \cref{lem:weaklyobjects}, we can deduce necessary conditions on $X$ to determine a terminal or an initial object on the whole Galois category.
\begin{corollary}\label[corollary]{cor:weakly_conditiions} Let $X$ a qcqs scheme. 
If the Galois category $\Gal(X)$ has a (weakly) terminal object, then $X$ is irreducible.
If $\Gal(X)$ has a (weakly) initial object, then $X$ is the spectrum of a local ring.
\end{corollary}
\begin{proof}
This follows as if $\Gal(X)$ has any of these properties, the underlying category $\Gal(X)(\ast)$ needs to satisfy them as well.
\end{proof}
So far, we know the following by combining the statements in \cref{rem:filtered_irreducible}, \cref{lem:weaklyobjects} and Remark \cref{rem:morphism_empty}:
\begin{itemize}
    \item[a)] If $\Gal(X)(\ast)$ is filtered, it has a weakly terminal object $t$ with $\Hom(t,x)=\emptyset$ if $x\neq t$ and $\Hom(t,t)$ only contains isomorphisms.
    \item[b)] If $\Gal(X)(\ast)$ is cofiltered, it has a weakly initial object $i$ with $\Hom(x,i)=\emptyset$ if $x\neq i$ and $\Hom(i,i)$ only contains isomorphisms.
    \end{itemize}
By a general argument on categories, the combination of these properties forces the weakly objects to be strict in their property.
We only state the terminal case as the argument for the initial one follows accordingly.
\begin{lemma}
Let $\Ccal$ be a filtered category with a weakly terminal object $t\in \Ccal$ such that $\Hom(t,x)=\emptyset$ if $x\neq t$ and $\Hom(t,t)$ only contains isomorphisms.
Then, $\Ccal$ already has a terminal object.
\end{lemma}
\begin{proof}
We need to show that all morphism sets into the weakly terminal object contain a unique morphism.
For an object $X\in \Ccal$, let $f,g\colon X \rightarrow t$ be any two parallel arrows to the weakly terminal object.
By $\Ccal$ being cofiltered, there is a morphism $h\colon t\rightarrow Y$ such that the pair of parallel arrows becomes equivalent after post-composition with $h$, i.e., we have $h\circ f=h\circ g$. Due to the additional assumptions on the morphisms sets, we not only have an identification $Y=t$ but also deduce that $h$ is an isomorphism. As $f,g$ become equivalent after post-composition with an isomorphism, we find $f=g$ showing that $t$ is indeed a terminal object.
\end{proof}
Altogether, we have argued the statement of \cref{lem:terminal_initial_underlying_cats} on the level of underlying categories.
That it also applies to whole $\Gal(X)$ will become apparent later. 
\subsection{Classification of Galois categories with initial or terminal object}\label[section]{sec:Classification}
For now, we will concentrate on the cases of $\Gal(X)$ having an initial or a terminal object and give subsequently a full classification.
\cref{cor:weakly_conditiions} already gives a hint in the right direction.
Particularly, we will show the following statements.
\begin{theorem}\label[theorem]{thm:Gal_ter_ini}
Let $X$ be a qcqs scheme.
We have the following equivalences:
\begin{itemize}
\item[a)] $\Gal(X)$ has a terminal object $\iff$ $X$ is irreducible and everywhere strictly local.
\item[b)] $\Gal(X)$ has an initial object $\iff$ $X$ is the spectrum of a strictly henselian local ring.
\end{itemize}
In particular, in both cases the schemes are condensed contractible, i.e., $\CondShape{X}=\ast$.
\end{theorem}
The notion of an \textit{everywhere strictly local} scheme is defined as follows.
\begin{definition}\label[definition]{def:everywhere_strictly:local}
Let $X$ be a scheme. 
Then $X$ is called \textit{everywhere strictly local} if all local rings $\mathcal{O}_{X,\xbar}$ at geometric points $\xbar\rightarrow X$ are strictly henselian.
\end{definition}
\subsubsection{Classification of Gal(X) having a terminal object}
We start with showing the following variant of part a) of \cref{thm:Gal_ter_ini}.
\begin{theorem}\label[theorem]{thm:Galterminal}
Let $X$ be a qcqs scheme.
Then the following are equivalent:
\begin{itemize}
\item[(i)] The scheme $X$ is irreducible and everywhere strictly local.
\item[(ii)] The condensed category $\Gal(X)$ has a terminal object.
\item[(iii)] The category of points $\Pt(X_{\et})$ has a terminal object.
\end{itemize}
\end{theorem}
\begin{remark}\label[remark]{rem:cofil_term_ini}
Note that the equivalence of part (ii) and (iii) gives the missing implication from \cref{prop:filt_term_ini}
\begin{center}
$\Gal(X)$ is cofiltered $\implies$ $\Gal(X)$ has a terminal object.
\end{center}
Indeed, $\Gal(X)$ being cofiltered implies $\Pt(X_{\et})$ to be cofiltered which was shown to be equivalent to $\Pt(X_{\et})$ having a terminal object in \cref{subsec:underlying_cats}.
\end{remark}
Before stating a complete proof of the theorem, we give an idea of the arguments and outsource different results needed.
Justification for the claim comes from the following: 
\begin{example}
Let $X$ be an absolutely integrally closed scheme, i.e., an integral normal scheme whose function field is algebraically closed.
In other words, $X$ is the strict normalization $Y^{(\ybar)}$ of a scheme $Y$ at a geometric point $\ybar\rightarrow Y$. 
In \cite[Corollary~12.4.5]{Exodromy}, it is claimed as an equivalence of profinite posets and proven in \cite[Corollary 6.17]{The_condensed_homotopy_type} for condensed categories that
$\Gal(X)=\Gal(Y)_{/\ybar}.$
Particularly, the Galois category $\Gal(X)$ has a terminal object corresponding to the geometric point $\ybar$.
\end{example}
As any absolute integrally closed scheme is everywhere strictly local \cite[Proposition 2.6]{MR3649361}, the final conclusion in the last example is a special case of what we are aiming to prove.
For a general everywhere strictly local scheme, we can fall back to the Zariski space.
\begin{proposition}\label[proposition]{prop:everwhstrictlylocal}
Let $X$ be an everywhere strictly local qcqs scheme.
Then for every $S\in \Extr$ we have an identification of categories
$$\Gal(X)(S)=\Hom_{qc}(S,|X|)\comma$$
where $\Hom_{qc}(S,|X|)$ is the set of quasicompact continuous maps $S \rightarrow |X|$ from $S$ to the underlying Zariski space of $X$ which inherits the structure of a category through the poset structure of $ |X|$.\\ More precisely, the \textit{poset} of quasicompact maps $f \colon S \to |X| $ is ordered by \textit{pointwise specialization:} it is $ f \leq g $ if and only if for all $ s \in S $, we have $ f(s) \in \overline{\{g(s)\}} $.
In particlular, $ \Gal(X)(\ast)$  recovers the specialization poset  $ \Zpos{X}$ of $ |X| $. 
\end{proposition}
\begin{proof}
For $X$ an everywhere strictly local scheme, the \'{e}tale and the Zariski topoi of $X$ agree \cite[Corollary 2.5]{MR3649361}.
Using this, we are reduced to algebraic morphisms between sheaf topoi of spectral topological spaces as 
$$\Gal(X)(S)\simeq\Fun^{*, \coh}(\Sh(|X|,\Set),\Sh(S,\Set))\period$$
Geometric morphisms of sheaf topoi on topological spaces are the same as morphisms of the underlying locales of opens of these spaces.
As proven in \cite[Proposition C.1.4.5]{MR1953060}, this refines to an equivalence of categories
\begin{align}\label{locale1}
\Fun_*(\Sh(S,\Set), \Sh(|X|,\Set))\simeq\Hom_{\mathrm{Loc}}(\mathrm{Ouv}(S), \mathrm{Ouv}(|X|))\comma
\end{align}
where a morphism $f\colon \mathrm{Ouv}(S) \rightarrow \mathrm{Ouv}(|X|)$ of the locales of open subsets corresponds to a morphism $$f^{-1}\colon \mathrm{Ouv}(|X|) \rightarrow \mathrm{Ouv}(S)$$ of frames. Then the set of morphisms on the right hand side in (\ref{locale1}) obtains the structure of a poset by the partial order of inclusions on the frames.
Indeed, we say that there exists a morphism $f \rightarrow g$ between two maps $f,g\colon \mathrm{Ouv}(S)\rightarrow \mathrm{Ouv}(|X|)$ if and only if for all $U\in  \mathrm{Ouv}(|X|)$ open subset of $|X|$ we have $f^{-1}(U)\subset g^{-1}(U)$.
Furthermore, it is known by \cite[Chapter IX, Section 3, Theorem 1]{MR1300636} that the functor $$\mathrm{Ouv}\colon \Top \rightarrow \mathrm{Loc}, \ X \mapsto \mathrm{Ouv}(X)$$ sending a topological space to its locale of opens has a right adjoint $\Pt\colon \mathrm{Loc}\rightarrow \Top$ and that the unit of the adjunction is a homeomorphism exactly on sober topological spaces \cite[Chapter IX, Section 3, Proposition 2]{MR1300636}.
As the underlying space of a scheme is sober, the adjunction gives bijective correspondences
\begin{align}\label{locale2}
\Hom_{\mathrm{Loc}}(\mathrm{Ouv}(S), \mathrm{Ouv}(|X|))&\simeq\Hom_{\Top}(S, \Pt(\mathrm{Ouv}(|X|))) \nonumber \\
&\simeq\Hom_{\Top}(S, |X|)\period
\end{align}
These extend to equivalences of categories by taking over the poset structure from the left hand side.
In more detail, there exists a morphism between two continuous maps $f, g\colon S\rightarrow |X|$ 
if and only if for all open subsets $U\subset |X|$ it is $f^{-1}(U) \subset g^{-1}(U)$. If we now take the opposite of this poset, we obtain a poset structure on $\Hom_{\Top}(S, |X|)$ given as follows: We have a morphism $f \rightarrow g$ or an order $f\leq g$ between two continuous maps $f, g\colon S\rightarrow |X|$ if and only if for all $s\in S$ it is $ f(s) \in \overline{\{g(s)\}} $. Indeed, the condition  $ f(s) \in \overline{\{g(s)\}} $ for all $s \in S$ is equivalent to the condition $g^{-1}(U) \subset f^{-1}(U)$ for all opens $U\subset |X|$.
Comparing the results of (\ref{locale1}) and (\ref{locale2}) with the claim in the proposition, it is only left to relate coherent geometric morphisms with quasi-compact maps of topological spaces.
That these two can in our situation be identified with each other is explained in \cite[Example 2.18]{Haine:1-localic}. Thus we finally end up with 
$$\Gal(X)(S)=\Hom_{qc}(S,|X|)\comma$$
where the poset structure on the right is inherited by the poset structure of $ \Zpos{X}$.
\end{proof}
\begin{corollary}
Let $X$ be an everywhere strictly local qcqs scheme.
Then its Galois category $\Gal(X)$ is contractible as a condensed category if and only if the  Zariski poset $\Zpos{X}$ is contractible as a poset.
\end{corollary}
\begin{proof}
This is a direct corollary of the proposition by realizing that the categorical structure of $\Gal(X)$ is exactly given by the poset structure of $ \Zpos{X}$.
\end{proof}
\cref{prop:everwhstrictlylocal} shows that for an everywhere strictly local scheme $X$, a maximal element of $ \Zpos{X}$, which corresponds by \cref{lem:XZarconditions} to a unique generic point of $X$, carries over to a terminal object of the whole Galois category $\Gal(X)$.
Similar observations can be made for local rings and initial objects. More precisely, we can state: 
\begin{corollary}\label[corollary]{cor:eslcase}
Let $X$ be an everywhere strictly local qcqs scheme.
Then, the Galois category $\Gal(X)$ has a terminal (respectively, an initial) object if and only if one of the equivalent conditions of \cref{lem:XZarconditions} is satisfied.
\end{corollary}
\begin{proof}
The equivalence with part a) of \cref{lem:XZarconditions} follows from  \cref{prop:everwhstrictlylocal}:
A morphism $t\colon K\rightarrow |X|$ is a terminal (respectively, $i\colon K\rightarrow |X|$ an initial) object of $\Gal(X)(S)$ if and only if $g(k)\leq t(k)$ (respectively, $i(k)\leq g(k)$) for all $g\in \Gal(X)(S)$ and $k\in K$.
Any constant map $c_{x}\colon K \rightarrow |X|, \ k\mapsto x$ of some $x\in |X|$ is contained in $\Gal(X)(S)$ as it is quasi-compact continuous.
Thus the terminal (respectively, initial) object of $\Gal(X)(S)$ corresponds to a constant map $c_{\mathrm{max}}$ (respectively, $c_{\mathrm{min}}$) of a maximal (respectively, minimal) element of $ \Zpos{X}$.
The choice of the constant map is independent of $K$.
\end{proof}
\begin{remark}
Any local and everywhere strictly local scheme is strictly henselian (additionally satisfying the property that not only the ring itself but also the local rings at all prime ideals are strictly henselian).
In this case, it is already known that the condensed homotopy type is trivial, compare \cref{ex:strictly_henselian_trivial}. 
\end{remark}
We will examine the whole situation for initial objects in \cref{thm:Galinitial} and stick here to the case of terminal objects.
In \cref{cor:eslcase}, we have seen that for an everywhere strictly local qcqs scheme $X$, irreducibility ensures the existence of a terminal object of $\Gal(X)$ proving the implication (i) to (ii) in  \cref{thm:Galterminal}.
So lets turn to the other direction.
We can extract the following more refined constraints on schemes $X$ whose Galois category $\Gal(X)$, and thus also the underlying category of points $\Gal(X)(\ast)=\Pt(X_{\et})$, has a terminal object.
\begin{proposition}\label[proposition]{prop:constraintsterm}
If the category of points $\Pt(X_{\acute{e}t})$ has a terminal object, then $X$ is an irreducible scheme satisfying the following additional properties:
\begin{itemize}
\item[a)] The residue field $\kappa(\eta)$ of the generic point $\eta\in X$ is separably closed.
\item[b)] For every local ring $\mathcal{O}_{X,x}$ at a point $x\in X$, the fiber of the generic point under the canonical map $\Spec(\mathcal{O}_{X,x}^{\mathrm{sh}}) \rightarrow \Spec(\mathcal{O}_{X,x})$ consists of only one point.
\end{itemize} 
\end{proposition}
\begin{proof}
We already know that existence of a terminal object for $\Pt(X_{\et})$ forces the scheme $X$ to be irreducible. The terminal object $t$ corresponds to its unique generic point, \cref{lem:weaklyobjects}.
By assumption, there exists a unique morphism $x\rightarrow t$ for every $x \in \Pt(X_{\et})$.
Particularly, this has to be satisfied for $t$ itself.
Recall that automorphisms of objects in $\Pt(X_{\et})$ are given by the absolute Galois groups of the residue fields of the corresponding geometric points of $X$.
Part a) then follows as the absolute Galois group of a field is trivial if and only if the field is separably closed. For part b), note that every preimage of $t$ under $\Spec(\mathcal{O}_{X,x}^{\mathrm{sh}}) \rightarrow \Spec(\mathcal{O}_{X,x})$ gives rise to a morphism $x\rightarrow t$ in $\Pt(X_{\et})$ as explained in \cref{rem:Pointsetaletopos2}.
Hence, the preimage of $t$ under the map is only allowed to consist of one point.
Identifying $t$ with the generic point yields the claim.
\end{proof}
We are very close to the assertion we are aiming at.
It remains to argue why a scheme  satisfying the properties in the last proposition is indeed everywhere strictly local.
The next proposition, which is based on a result by Schr\"oer, is the key to this conclusion. 
\begin{proposition}\label[proposition]{prop:eslgeomuni}
Let $X$ be an irreducible qcqs scheme.
The following assertions are equivalent:
\begin{itemize}
\item[(i)] The scheme is everywhere strictly local, i.e., all local rings of $X$ are strictly henselian.
\item[(ii)] The scheme is unibranch and the residue field $\kappa(\eta)$ of the generic point $\eta \in X$ is separably closed.
\item[(iii)] The scheme is geometrically unibranch and the residue field $\kappa(\eta)$ of the generic point $\eta \in X$ is separably closed.
\end{itemize}
\end{proposition}
\begin{proof}
The equivalence $(i) \Leftrightarrow (ii)$ is proven in \cite[Proposition 2.6]{MR3649361}.
We show that it makes no difference if we replace the unibranch condition in (ii) by the condition of being geometrically unibranch.
Recall that a scheme is called (geometrically) unibranch if all local rings satisfy this property \cite[\href{https://stacks.math.columbia.edu/tag/0BQ1 }{Tag 0BQ1}]{stacksproject}.
As geometrically unibranch implies unibranch by definition, we only need to argue why $(i) \implies (iii)$.
Assuming (i), we know that the residue fields at all local rings are separably closed.
Hence, all extensions of these residue fields have to be purely inseparably. As we know that (i) implies $X$ to be unibranch, geometrically unibranch follows almost by definition:
A local ring $A$ is geometrically unibranch if it is unibranch and the residue field of the integral closure $A'$ of the reduction $A_{\mathrm{red}}$ in its fraction field is purely inseparable over the residue field of $A$ \cite[\href{https://stacks.math.columbia.edu/tag/0BPZ }{Tag 0BPZ}]{stacksproject}.
This is true for all local rings by the previous conclusions.
\end{proof}
According to this proposition, we are reduced to showing that the scheme $X$ is geometrically unibranch, i.e., the strict henselizations $\Spec(\mathcal{O}_{X,\xbar}^{\mathrm{sh}})$ of all local rings at geometric points $\xbar\rightarrow X$ are irreducible \cite[\href{https://stacks.math.columbia.edu/tag/06DM }{Tag 06DM}]{stacksproject}.
In fact, this turns out to be true by applying the following corollary to $\Spec(\mathcal{O}_{X,x}^{\mathrm{sh}})$.
It relies on the more general result of the lemma below.
\begin{corollary}\label[corollary]{cor:henselirred}
Let $A$ be a local ring with irreducible spectrum and strict henselization $A^{\mathrm{sh}}$.
If the fiber of the generic point under the canonical map $\Spec(A^{\mathrm{sh}}) \rightarrow \Spec(A)$ consists of only one point, the spectrum of $A^{\mathrm{sh}}$ is also irreducible.
\end{corollary}
\begin{lemma}
Let $Y$ be an irreducible qcqs scheme with $\eta \in Y$ its generic point and $f\colon X\rightarrow Y$ a flat morphism of schemes. 
Then $f$ maps every maximal point of $X$ to $\eta$.
\end{lemma}
\begin{proof}
We use \cite[Lemma 14.42]{AG1}, which tells us that taking the inverse image of $\{\eta\}$ under $f$ is compatible with taking closures, i.e., we have
$f^{-1}(\etabar)=\overline{f^{-1}(\eta)} $.
Note that the Lemma can be applied to our prerequisites as flat morphisms of schemes are generizing~\cite[Lemma~14.9]{AG1} and the canonical map $\Spec(\kappa(\eta))\rightarrow X$ mapping to $\eta$ exhibits the one point set $\{\eta\}$ as a pro-constructible set.
As $\eta$ is the generic point, the closure of the preimage of $\eta$ has to be the whole space $X$.
Thus, all maximal points of $X$, which are the unique generic points of the irreducible components, have to be contained in the preimage, i.e., they are mapped to $\eta$ under $f$. 
\end{proof}
\begin{proof}[Proof of \cref{cor:henselirred}]
The canonical map $\Spec(A^{\mathrm{sh}}) \rightarrow \Spec(A)$ is faithfully flat and, in particular, a flat morphism mapping to an irreducible qcqs scheme by our assumptions on $A$.
By the lemma above, the fiber of the generic point contains all maximal points of $\Spec(A^{\mathrm{sh}})$.
As there is only one point in the preimage assumed, we can deduce that the spectrum of the strict henselization has a unique maximal point which means nothing else than that it its irreducible.
\end{proof}
Roughly speaking, we have seen the following: 
\begin{itemize}
    \item[1.] Irreducibility of $X$ provides a candidate for the terminal object of $\Gal(X)$ given by the unique generic point of $X$.
    \item[2.] To secure that there are no non-trivial automorphisms of objects in $\Gal(X)$ that do not contract to the identity on the terminal object under the (condensed) classifying anima, it is necessary to add the condition on $X$ to be everywhere strictly local.
\end{itemize}
\begin{proof}[Complete Proof of \cref{thm:Galterminal}] We put all the pieces together. 
Let $X$ be an irreducible and everywhere strictly local qcqs scheme.
Then its Galois category has a terminal object by \cref{cor:eslcase}.
For the other directions, we assume that $\Gal(X)$ has a terminal object.
Then, the underlying category of points $\Pt(X_{\et})$ has a terminal object and \cref{prop:constraintsterm} shows that $X$ is an irreducible scheme satisfying additional properties which imply that $X$ is already everywhere strictly local:
Property b) in \cref{prop:constraintsterm} combined with \cref{cor:henselirred} tells us that for every point $x\in X$ the strict henselizations $\mathcal{O}_{X,x}^{\mathrm{sh}}$ is irreducible.
As this means nothing else than that $X$ is geometrically unibranch \cite[\href{https://stacks.math.columbia.edu/tag/06DM }{Tag 06DM}]{stacksproject}, the claim follows from the equivalence of (i) and (iii) in \cref{prop:eslgeomuni} under consideration of property a) in \cref{prop:constraintsterm}.
\end{proof}
\begin{remark}
\cref{thm:Galterminal} indeed shows that the Galois category $\Gal(X)$ only hits certain types of categories.
Whenever it has a terminal object, it has level-wise the structure of a poset and there do not appear non-trivial automorphisms of points.
\end{remark}
\subsubsection{Classification of Gal(X) having an initial object}
We have explained in great detail what it means for a scheme to have a Galois category with trivial object.
We can make similar considerations for the case of an initial object.
\begin{theorem}\label[theorem]{thm:Galinitial}
Let $X$ be a qcqs scheme.
Then the following are equivalent:
\begin{itemize}
\item[(i)] The scheme $X$ is the spectrum of a strictly henselian local ring.
\item[(ii)] The condensed category $\Gal(X)$ has an initial object.
\item[(iii)] The category of points $\Pt(X_{\et})$ has an initial object.
\end{itemize}
\end{theorem}
Again, we outsource several steps of the proof.
The main work for $(i)\implies (ii)$ was already done in \cref{sec:via_weakly_proobjects}. 
Namely, it is basically the content of \cref{prop:alter_Gal_prop}.
\begin{corollary}\label[corollary]{cor:stricthensel_initial}
If we have given $X=\Spec(A)$ for $A$ a strictly henselian local ring, then $\Gal(X)$ has an initial object. 
\end{corollary}
\begin{proof}
This was proven in \cref{prop:alter_Gal_prop} as we have shown that for every $S=\lim_{i \in I} S_i\in \Extr$, the category $\Gal(X)(S)$ has an initial object given by the ring 
$\colim_{i \in I^{\op}}\prod_{s\in S_i} A$.
\end{proof}
Combining the two subsequent propositions gives the implication $(iii)\implies (i)$.
As the existence of an initial object of $\Gal(X)$ presupposes that this also applies to $\Pt(X_{\et})$, we have $(ii) \implies (iii)$ for free.
Moreover, note that we have the counterpart of \cref{rem:cofil_term_ini} meaning that $\Gal(X)$ has an initial object if and only if it is filtered.
\begin{proposition}\label[proposition]{prop:initalconditions}
Let $X$ be a qcqs scheme.
If the category of points $\Pt(X_{\acute{e}t})$ has an initial object, $X$ is the spectrum of a local ring $A$ satisfying the following additional property:
\begin{itemize}
    \item[$(\ast)$] The canonical map $\Spec(A^{\mathrm{sh}}) \rightarrow \Spec(A)$ is a universal homeomorphism.
\end{itemize}
\end{proposition}
\begin{proof}
Again by \cref{lem:weaklyobjects}, $X$ corresponds to a local ring $A$ and the initial object $i \in \Pt(X_{\et})$ is given by its unique maximal ideal.
The morphism $\Spec(A^{\mathrm{sh}}) \rightarrow \Spec(A)$, where $A^{\mathrm{sh}}$ is the strict henselization of the local ring $A$, is quasi-compact faithfully flat and thus a quotient map of topological spaces \cite[Proposition 14.39]{AG1}.
As the property to be faithfully flat quasi-compact is stable under base change, also every base change of $\Spec(A^{\mathrm{sh}}) \rightarrow \Spec(A)$ is a quotient map \cite[Corollary 14.43]{AG1}.
Since moreover a bijective quotient map is a homeomorphism and surjectivity is stable under base change, to show ($\pt$) we are reduced to proving that $\Spec(A^{\mathrm{sh}}) \rightarrow \Spec(A)$ is universally injective.
Referring to \cite[Proposition 4.35]{AG1}, this follows from the following two steps.
\begin{itemize}
    \item[\textbf{1.}] We show that the map of schemes $\Spec(A^{\mathrm{sh}}) \rightarrow \Spec(A)$ is injective:
    For a point $x\in \Spec(A)$, every object in the fiber of $x$ under $\Spec(A^{\mathrm{sh}}) \rightarrow \Spec(A)$ determines at least one morphism $i \rightarrow x$.
    As there is assumed to exist only one such morphism, the fiber has to consist of exactly one point.
    As the argument works for every $x\in \Spec(A)$, we can deduce injectivity. 
    \item[\textbf{2.}] For a point $x\in \Spec(A)$, we denote its preimage under $\Spec(A^{\mathrm{sh}}) \rightarrow \Spec(A)$ by $x'$.
    Then by \cite[\href{https://stacks.math.columbia.edu/tag/092R}{Lemma 092R}]{stacksproject}, the field extension $\kappa(x')$ is separable algebraic over $\kappa(x)$.
    We show that it is moreover purely inseparable over $\kappa(x)$.
    Therefore, it is enough to show that there exists exactly one $\kappa(x)$-embedding of $\kappa(x')$ into an algebraic closure $\Omega$ \cite[Corollary B.102. and B.99. (2)]{AG1}.
    As $\kappa(x')$ is already known to be separable, every $\kappa(x)$-embedding into an algebraic closure factors over the separable closure  $\kappa(x)^{\mathrm{sep}}\rightarrow \Omega$.
    A morphism $i \rightarrow x$ from the initial object $i$ in $\Pt(X_{\et})$ corresponds to a lift $\Spec(\kappa(x)^{\mathrm{sep}})\rightarrow \Spec(A^{\mathrm{sh}})$ such that the triangle
    \[
  \begin{tikzcd}[row sep=small]
   \Spec(\kappa(x)^{\mathrm{sep}}) \arrow[dotted]{r} \arrow[swap]{dr} & \Spec(A^{\mathrm{sh}}) \arrow{d} \\
     & \Spec(A)
  \end{tikzcd}
    \]
commutes.
In particular, the lift factors over the unique preimage $x'$ of $x$ as a map
\begin{align}\label{factorization}
\Spec(\kappa(x)^{\mathrm{sep}})\rightarrow \Spec(\kappa(x'))(\rightarrow \Spec(A^{\mathrm{sh}}))
\end{align} 
that is compatible with the maps to $\Spec(\kappa(x))$.
In other words, (\ref{factorization}) corresponds to a $\kappa(x)$-embedding $\kappa(x')\rightarrow \kappa(x)^{\mathrm{sep}}$.
By uniqueness of $i\rightarrow x$, there is exactly one choice for the factorization (\ref{factorization}).
This means, there is a unique $\kappa(x)$-embedding $\kappa(x')\rightarrow \kappa(x)^{\mathrm{sep}}$. 
\end{itemize}
Note that from Part 2. it even follows that $\kappa(x)\simeq\kappa(x')=\kappa(x)^{\mathrm{sep}}$ as $\kappa(x')$ is not only a separable but also a purely inseparable field extension of $\kappa(x)$.
\end{proof}
\begin{proposition}\label[proposition]{prop:localring_homeo}
A local ring $A$ which satisfies condition $(\ast)$ in \cref{prop:initalconditions} is strictly henselian.
Indeed, we have $A\simeq A^{\mathrm{sh}}.$
\end{proposition}
\begin{proof}
Let $A\rightarrow A^{\mathrm{sh}}$ be the canonical morphism to the strict henselization of $A$.
We first show that condition $(\ast)$ in \cref{prop:initalconditions} implies that $A$ is henselian. 
Then the claim follows from the isomorphism of residue fields mentioned in the last sentence in the proof of \cref{prop:initalconditions}.
By \cite[\href{https://stacks.math.columbia.edu/tag/09XI}{Lemma 09XI}]{stacksproject}, henselian local rings can be characterized as follows: 
A local ring $A$ with a maximal ideal $m$ is henselian if and only if for every finite $A$-algebra $R$ the map $R\rightarrow R/mR$ induces a bijection on idempotent elements
\begin{align}\label{idempo}
\mathrm{Idem}(R)=\mathrm{Idem}(R/mR)\period
\end{align}
So let $A\rightarrow R$ be a finite morphism of rings.
Then, $R\otimes_A A^{\mathrm{sh}}$ is a finite $A^{\mathrm{sh}}$-module as finiteness is preserved under base change.
Since $A^{\mathrm{sh}}$ is henselian, we have by the characterization above a bijection
$$\mathrm{Idem}(R\otimes_A A^{\mathrm{sh}})=\mathrm{Idem}((R\otimes_A A^{\mathrm{sh}})/m_{\mathrm{sh}}(R\otimes_A A^{\mathrm{sh}}))=\mathrm{Idem}(R\otimes_A A^{\mathrm{sh}}/m_{\mathrm{sh}})\comma$$ where $m_{\mathrm{sh}}$ denotes the maximal ideal of $A^{\mathrm{sh}}$.
The maximal ideal $m_{\mathrm{sh}}$ can be identified with the ideal generated by the image of the maximal ideal $m \subset A$ under $A\rightarrow A^{\mathrm{sh}}$.
Thus, we have $A^{\mathrm{sh}}/m_{\mathrm{sh}}=A/m\otimes_A A^{\mathrm{sh}}$ and can further conclude 
$$\mathrm{Idem}(R\otimes_A A^{\mathrm{sh}})=\mathrm{Idem}(R\otimes_A A/m\otimes_A A^{\mathrm{sh}})$$ by replacing $A^{\mathrm{sh}}/m_{\mathrm{sh}}$.
So far, we did not make use of condition $(\ast)$.
This comes into play now. 
As $A\rightarrow A^{\mathrm{sh}}$ induces by condition $(\ast)$ a universal homeomorphism on prime spectra, also the canonical maps $R\rightarrow R\otimes_A A^{\mathrm{sh}}$ and $R\otimes_A A/m \rightarrow R\otimes_A A/m \otimes_A A^{\mathrm{sh}}$ induce homeomorphisms.
Recall that idempotents of a ring are in one to one correspondence with open and closed subsets of the prime spectrum \cite[\href{https://stacks.math.columbia.edu/tag/00EE}{Lemma 00EE}]{stacksproject}.
Using this identification, the previously mentioned homeomorphisms induce bijections 
\begin{align*}
 \mathrm{Idem}(R)&=\mathrm{Idem}(R\otimes_A A^{\mathrm{sh}}) \\
 \mathrm{Idem}(R/mR)&=\mathrm{Idem}(R\otimes_A A/m)=\mathrm{Idem}(R\otimes_A A/m \otimes_A A^{\mathrm{sh}})\period
\end{align*}
Combining these with the above shown bijections of idempotents leads to (\ref{idempo}) for every finite $A$-module $R$ and we can conclude that $A$ is henselian.
As the residue field of $A^{\mathrm{sh}}$ is not only separable but also purely inseparable over the residue field of $A$ (Part $2.$ in the proof of \cref{prop:localring_homeo}), we find an isomorphism of residue fields showing that $A$ coincides with its strict henselization.
\end{proof}
Note that \cite[\href{https://stacks.math.columbia.edu/tag/0F6V}{Lemma 0F6V}]{stacksproject} gives another proof for the result using that the morphism $\Spec(A^{\mathrm{sh}})\rightarrow \Spec(A)$ is a weakly \'etale universal homeomorphism.
In fact, the subsequent more general result would have been sufficient as well.
\begin{lemma}
Let $A \rightarrow B$ be a map of rings such that the induced morphism $$\Spec(B)\rightarrow \Spec(A)$$ is a universal homeomorphism of schemes.
Then, if $B$ is a strictly henselian local ring also $A$ is a strictly henselian local ring.
\end{lemma}
\begin{proof}
First note that $A$ automatically is a local ring due to the homeomorphism.
For the henselian property, we can argue as in the proof of the last lemma via the characterization of henselity by idempotents.\footnote{Actually, the argument also works if we change the roles of $A$ and $B$.}
It remains to prove that the residue field $\kappa_A$ of $A$ is separably closed. Under the assumption, the separably closed residue field $\kappa_B$ is a purely inseparable field extension of the residue field $\kappa_A$ \cite[Proposition 4.35]{AG1} and needs to contain the separable closure of $\kappa_A$.
By pure inseparability, there cannot exist a proper intermediate separable extension \cite[\href{https://stacks.math.columbia.edu/tag/030K}{Lemma 030K}]{stacksproject}.
Hence, it follows that $\kappa_A$ is already separably closed.
\end{proof}
We have seen a classification of schemes whose condensed homotopy type is trivial.
We conclude with a remark on universal homeomorphisms and trivial condensed homotopy types.
\begin{remark}
Invariance under universal homeomorphisms of the (pro-)\'etale topos of a scheme $X$ (see, e.g., \cite[Lemma 5.4.2]{BhattScholzeProetale}) implies the same for the Galois category $\Gal(X)$.
Hence, the class of schemes for which we showed triviality of the condensed homotopy type in this section does not enlarge under the formation of universal homeomorphisms.
\end{remark}


\DeclareFieldFormat{labelalphawidth}{#1}
\DeclareFieldFormat{shorthandwidth}{#1}
\printbibliography[heading=references]

@misc{reconstruction_etale_topoi,
      title={Reconstruction of schemes from their \'{e}tale topoi}, 
      author={Magnus Carlson and Peter J. Haine and Sebastian Wolf},
      year={2024},
      eprint={2407.19920},
      archivePrefix={arXiv},
      primaryClass={math.AG},
      url={https://arxiv.org/abs/2407.19920}, 
}

@book{classfieldtheory,
language = {eng},
publisher = {Springer International Publishing},
series = {Universitext},
title = {Galois Cohomology and Class Field Theory},
url = {https://doi.org/10.1007/978-3-030-43901-9},
edition = {1st ed. 2020},
isbn = {9783030439019},
author = {Harari, David},
year = {2020},
}

@book {Localfields,
    AUTHOR = {Serre. Jean-Pierre},
     TITLE = {Local fields},
    SERIES = {Graduate Texts in Mathematics},
 PUBLISHER = {Springer-Verlag, New York},
      YEAR = {1980},
     PAGES = {241},
      ISBN = {978-0-387-90424-5},
}

@book{valuedfields,
publisher = {Springer},
series = {Springer monographs in mathematics},
title = {Valued fields},
url = {https://digitale-objekte.hbz-nrw.de/storage2/2018/02/01/file_51/7512298.pdf},
address = {Berlin [u.a},
isbn = {354024221X},
author = {Engler, Antonio J. and Prestel, Alexander},
year = {2005},
}

@book{Algebraicnumbertheory,
language = {eng},
publisher = {Springer},
series = {Grundlehren der mathematischen Wissenschaften 322},
title = {Algebraic number theory},
address = {Berlin, Heidelberg, New York},
isbn = {3540653996},
author = {Neukirch, Jürgen},
year = {1999},
}

@article {MR4574234,
       IDS = {arXiv:2012.10502,zbMATH07671238},
    AUTHOR = {Wolf, Sebastian},
     TITLE = {The pro-étale topos as a category of pyknotic presheaves},
   JOURNAL = {Doc. Math.},
  FJOURNAL = {Documenta Mathematica},
    VOLUME = {27},
      YEAR = {2022},
     PAGES = {2067--2106},
      ISSN = {1431-0635},
   MRCLASS = {18F10 (14F20)},
  MRNUMBER = {4574234},
      NOTE = {\href{https://arxiv.org/abs/2012.10502}{\nolinkurl{arXiv:2012.10502}}},
}

@unpublished{arXiv:2012.10502,
    Author = {Mair, Catrin},
    Month = {May},
    Note = {\href{https://arxiv.org/abs/2012.10502}{\nolinkurl{arXiv:2012.10502}}},
    Title = {Animated Condensed Sets and Their Homotopy Groups},
    Year = {2021}, 
}

@unpublished{Classifyinganima,
      title={Classifying anima of condensed $\infty$-categories of points}, 
      author={Peter J. Haine},
      year={2026},
      eprint={2602.21330},
      archivePrefix={arXiv},
      primaryClass={math.CT},
      url={https://arxiv.org/abs/2602.21330}, 
}

@article{GABBER20154667,
title = {Points in algebraic geometry},
journal = {Journal of Pure and Applied Algebra},
volume = {219},
number = {10},
pages = {4667-4680},
year = {2015},
issn = {0022-4049},
doi = {https://doi.org/10.1016/j.jpaa.2015.03.001},
url = {https://www.sciencedirect.com/science/article/pii/S0022404915000730},
author = {Ofer Gabber and Shane Kelly},
}

@article {MR4609461,
       IDS = {hemo2023constructible_sheaves_schemes},
    AUTHOR = {Hemo, Tamir and Richarz, Timo and Scholbach, Jakob},
     TITLE = {Constructible sheaves on schemes},
   JOURNAL = {Adv. Math.},
  FJOURNAL = {Advances in Mathematics},
    VOLUME = {429},
      YEAR = {2023},
     PAGES = {Paper No. 109179, 46},
      ISSN = {0001-8708,1090-2082},
   MRCLASS = {14F08 (14E20 14F06 14F20)},
  MRNUMBER = {4609461},
MRREVIEWER = {Damian\ M.\ Maingi},
       DOI = {10.1016/j.aim.2023.109179},
       URL = {https://doi.org/10.1016/j.aim.2023.109179},
      NOTE = {\href{https://arxiv.org/abs/2305.18131}{\nolinkurl{arXiv:2305.18131}}},
}

@book {MR1300636,
    AUTHOR = {Mac Lane, Saunders and Moerdijk, Ieke},
     TITLE = {Sheaves in geometry and logic},
    SERIES = {Universitext},
      NOTE = {A first introduction to topos theory,
              Corrected reprint of the 1992 edition},
 PUBLISHER = {Springer-Verlag, New York},
      YEAR = {1994},
     PAGES = {xii+629},
      ISBN = {0-387-97710-4},
   MRCLASS = {03G30 (18B25 54B40)},
  MRNUMBER = {1300636},
MRREVIEWER = {M. Makkai},
}

@unpublished{pyknoticI,
  Author = {Barwick, Clark and Haine, Peter J.},
  Month = {April},
  Note = {\href{https://arxiv.org/abs/1904.09966}{\nolinkurl{arXiv:1904.09966}}},
  Title = {Pyknotic objects, {I}. {B}asic notions},
  Year = {2019}}

@unpublished{Scholze:analyticnotes,
    Author = {Scholze, Peter},
    Month = {October},
    Note = {Lecture notes available at \href{https://www.math.uni-bonn.de/people/scholze/Analytic.pdf}{\nolinkurl{math.uni-bonn.de/people/scholze/Analytic.pdf}}},
    Title = {Lectures on Analytic Geometry},
    Year = {2019}}

@unpublished{Scholze:condensednotes,
	Author = {Scholze, Peter},
	Month = {April},
	Note = {Lecture notes available at \href{http://www.math.uni-bonn.de/people/scholze/Condensed.pdf}{\nolinkurl{math.uni-bonn.de/people/scholze/Condensed.pdf}}},
	Title = {Lectures on Condensed Mathematics},
	Year = {2019}}

@unpublished{Haine:1-localic,
	Author = {Haine, Peter J.},
	Month = {April},
	Note = {\href{https://arxiv.org/abs/1904.01877}{\nolinkurl{arXiv:1904.01877}}},
	Title = {On coherent topoi \& coherent {$1$}-localic {$\infty$}-topoi},
	Year = {2019}}

@misc{Kerodon,
	Author = {Lurie, Jacob},
	Howpublished = {\href{https://kerodon.net}{\nolinkurl{kerodon.net}}},
	Shorthand = {Ker},
	Title = {Kerodon},
	Year = {2025}}

@unpublished{Ultracategories,
	Author = {Lurie, Jacob},
	Note = {\href{http://www.math.ias.edu/~lurie/papers/Conceptual.pdf}{\nolinkurl{math.ias.edu/~lurie/papers/Conceptual.pdf}}},
	Title = {Ultracategories},
	Year = {2018}}

@article {MR3248993,
    AUTHOR = {Holschbach, Armin and Schmidt, Johannes and Stix, Jakob},
     TITLE = {Étale contractible varieties in positive characteristic},
   JOURNAL = {Algebra Number Theory},
  FJOURNAL = {Algebra \& Number Theory},
    VOLUME = {8},
      YEAR = {2014},
    NUMBER = {4},
     PAGES = {1037--1044},
      ISSN = {1937-0652,1944-7833},
   MRCLASS = {14F35},
  MRNUMBER = {3248993},
MRREVIEWER = {Junwu\ Tu},
       DOI = {10.2140/ant.2014.8.1037},
       URL = {https://doi.org/10.2140/ant.2014.8.1037},
      NOTE = {\href{https://arxiv.org/abs/1310.2784}{\nolinkurl{arXiv:1310.2784}}},
}

@book {MR245577,
       IDS = {MR0245577},
    AUTHOR = {Artin, M. and Mazur, B.},
     TITLE = {Étale homotopy},
    SERIES = {Lecture Notes in Mathematics},
    VOLUME = {No. 100},
 PUBLISHER = {Springer-Verlag, Berlin-New York},
      YEAR = {1969},
     PAGES = {iii+169},
   MRCLASS = {14.55 (18.00)},
  MRNUMBER = {245577},
MRREVIEWER = {James\ Milne},
}

@book {MR676809,
    AUTHOR = {Friedlander, Eric M.},
     TITLE = {Étale homotopy of simplicial schemes},
    SERIES = {Annals of Mathematics Studies},
    VOLUME = {No. 104},
 PUBLISHER = {Princeton University Press, Princeton, NJ; University of Tokyo
              Press, Tokyo},
      YEAR = {1982},
     PAGES = {vii+190},
      ISBN = {0-691-08288-X},
   MRCLASS = {55P99 (14F35 55-02)},
  MRNUMBER = {676809},
MRREVIEWER = {V.\ P.\ Snaith},
}

@book{HTT,
	Address = {Princeton, NJ},
	Author = {Lurie, Jacob},
	Mrclass = {18-02 (18B25 18E35 18G30 18G55 55U40)},
	Mrnumber = {2522659 (2010j:18001)},
	Mrreviewer = {Mark Hovey},
	Pages = {xviii+925},
	Publisher = {Princeton University Press},
	Series = {Annals of Mathematics Studies},
	Shorthand = {HTT},
	Title = {Higher topos theory},
	Volume = {170},
	Year = {2009}}

@unpublished{SAG,
	Author = {Lurie, Jacob},
	Note = {\href{http://www.math.ias.edu/~lurie/papers/SAG-rootfile.pdf}{\nolinkurl{math.ias.edu/~lurie/papers/SAG-rootfile.pdf}}},
	Shorthand = {SAG},
	Title = {Spectral Algebraic Geometry},
	Year = {2018}}

@misc{stacksproject,
       IDS = {stacks-project},
	Author = {{The Stacks Project Authors}},
	Howpublished = {\href{http://stacks.math.columbia.edu}{\nolinkurl{stacks.math.columbia.edu}}},
	Shorthand = {STK},
	Title = {Stacks Project},
	Year = {2025},
}

@book{SGA4ii,
	Address = {Berlin},
	Mrclass = {14-06},
	Mrnumber = {50 \#7131},
	Pages = {iv+\xspace418},
	Publisher = {Springer-Verlag},
	Series = {Séminaire de Géométrie Algébrique du Bois Marie 1963--64 (SGA 4). Dirigé par M. Artin, A. Grothendieck, J.-L. Verdier. Avec la collaboration de N. Bourbaki, P. Deligne, B. Saint--Donat. Lecture Notes in Mathematics, Vol. 270},
	Shorthand = {SGA 4\textsubscript{\textsc{ii}}},
	Title = {Théorie des topos et cohomologie étale des schémas. {T}ome 2},
	Year = {1963--64}}

@article {MR3379634,
       IDS = {BhattScholzeProEtale,BhattScholzeProetale},
    AUTHOR = {Bhatt, Bhargav and Scholze, Peter},
     TITLE = {The pro-étale topology for schemes},
   JOURNAL = {Astérisque},
  FJOURNAL = {Astérisque},
    NUMBER = {369},
      YEAR = {2015},
     PAGES = {99--201},
      ISSN = {0303-1179,2492-5926},
      ISBN = {978-2-85629-805-3},
   MRCLASS = {14F05 (14F20 14F35 14H30 18B25)},
  MRNUMBER = {3379634},
MRREVIEWER = {Pieter\ Belmans},
      NOTE = {\href{https://arxiv.org/abs/1309.1198}{\nolinkurl{arXiv:1309.1198}}},
}

@article {HoyoisGalois,
    AUTHOR = {Hoyois, Marc},
     TITLE = {Higher {G}alois theory},
   JOURNAL = {J. Pure Appl. Algebra},
  FJOURNAL = {Journal of Pure and Applied Algebra},
    VOLUME = {222},
      YEAR = {2018},
    NUMBER = {7},
     PAGES = {1859--1877},
      ISSN = {0022-4049,1873-1376},
   MRCLASS = {18B25 (54C56 55P55)},
  MRNUMBER = {3763287},
MRREVIEWER = {Martin\ Szyld},
       DOI = {10.1016/j.jpaa.2017.08.010},
       URL = {https://doi.org/10.1016/j.jpaa.2017.08.010},
      NOTE = {\href{https://arxiv.org/abs/1506.07155}{\nolinkurl{arXiv:1506.07155}}},
}

@Book{Johnstone1982StoneSpaces,
  author    = {Peter T. Johnstone},
  publisher = {Cambridge University Press},
  title     = {Stone spaces},
  year      = {1982},
  address   = {Cambridge; New York},
  number    = {3},
  series    = {Cambridge studies in advanced mathematics},
}

@book{MR1953060,
	Author = {Johnstone, P. T.},
	Isbn = {0-19-853425-6},
	Mrclass = {18B25 (18-02)},
	Mrnumber = {1953060 (2003k:18005)},
	Mrreviewer = {Colin McLarty},
	Pages = {xxii+468+71},
	Publisher = {The Clarendon Press, Oxford University Press, New York},
	Series = {Oxford Logic Guides},
	Title = {Sketches of an elephant: a topos theory compendium. {V}ol. 1},
	Volume = {43},
	Year = {2002}}

@article {MR3518559,
    AUTHOR = {Glasman, Saul},
     TITLE = {A spectrum-level {H}odge filtration on topological
              {H}ochschild homology},
   JOURNAL = {Selecta Math. (N.S.)},
  FJOURNAL = {Selecta Mathematica. New Series},
    VOLUME = {22},
      YEAR = {2016},
    NUMBER = {3},
     PAGES = {1583--1612},
      ISSN = {1022-1824},
   MRCLASS = {18F25 (16E40 55P43)},
  MRNUMBER = {3518559},
MRREVIEWER = {Markus Szymik},
       DOI = {10.1007/s00029-016-0228-z},
       URL = {https://doi.org/10.1007/s00029-016-0228-z},
      NOTE = {\href{https://arxiv.org/abs/1408.3065}{\nolinkurl{arXiv:1408.3065}}},
}

@article {MR3649361,
    AUTHOR = {Schröer, Stefan},
     TITLE = {Geometry on totally separably closed schemes},
   JOURNAL = {Algebra Number Theory},
  FJOURNAL = {Algebra \& Number Theory},
    VOLUME = {11},
      YEAR = {2017},
    NUMBER = {3},
     PAGES = {537--582},
      ISSN = {1937-0652,1944-7833},
   MRCLASS = {14F20 (13B22 13J15 14E05)},
  MRNUMBER = {3649361},
MRREVIEWER = {Christian\ Liedtke},
       DOI = {10.2140/ant.2017.11.537},
       URL = {https://doi.org/10.2140/ant.2017.11.537},
      NOTE = {\href{https://arxiv.org/abs/1503.02891}{\nolinkurl{arXiv:1503.02891}}},
}

@unpublished{Exodromy,
    IDS = {exodromy},
	Author = {Barwick, Clark and Glasman, Saul and Haine, Peter J.},
	Note = {\href{https://arxiv.org/abs/1807.03281v7}{\nolinkurl{arXiv:1807.03281v7}}},
	Title = {Exodromy},
	Year = {2020}, 
}

@unpublished{CatrinsThesis,
	Author = {Mair, Catrin},
    Note = {PhD Thesis available at \href{https://tuprints.ulb.tu-darmstadt.de/29744/}{\nolinkurl{https://tuprints.ulb.tu-darmstadt.de/29744/}}},
	Title = {The Role of Condensed Mathematics in Homotopy Theory},
    Month ={April},
	Year = {2025},
    }

@misc{The_condensed_homotopy_type,
      title={The condensed homotopy type of a scheme}, 
      author={Peter J. Haine and Tim Holzschuh and Marcin Lara and Catrin Mair and Louis Martini and Sebastian Wolf},
      year={2025},
      eprint={2510.07443},
      archivePrefix={arXiv},
      primaryClass={math.AG},
      url={https://arxiv.org/abs/2510.07443}, 
}


\bigskip

\noindent \textsc{Catrin Mair, Universität Münster, Einsteinstrasse 62, 48149 Münster, Germany}
\medskip

\end{document}